\documentclass[aop]{imsart}

\usepackage{appendix}
      \usepackage{amssymb}
\usepackage{color}
\usepackage{dsfont}
\usepackage[pdftex]{graphicx}
\usepackage{amsmath}
\RequirePackage[OT1]{fontenc}
\RequirePackage{amsthm,amsmath}
\RequirePackage[numbers]{natbib}
\RequirePackage[colorlinks,citecolor=blue,urlcolor=blue]{hyperref}

\arxiv{arXiv:1711.01258}

\startlocaldefs
      \numberwithin{equation}{section}
\theoremstyle{plain}
\newtheorem{theorem}{Theorem}[section]
      \newtheorem{lemma}[theorem]{Lemma}
      \newtheorem{proposition}[theorem]{Proposition}
      \newtheorem{corollary}[theorem]{Corollary}
      
      \newtheorem{definition}[theorem]{Definition}
      
      \newtheorem{remark}[theorem]{Remark}
\endlocaldefs

\begin{document}

\begin{frontmatter}

% "Title of the paper"
\title{On the Transient (T) condition for Random Walk in Mixing Environment}
\runtitle{Random Walk in Mixing Environment}

% indicate corresponding author with \corref{}
% \author{\fnms{John} \snm{Smith}\corref{}\ead[label=e1]{smith@foo.com}\thanksref{t1}}
% \thankstext{t1}{Thanks to somebody}
% \address{line 1\\ line 2\\ printead{e1}}
% \affiliation{Some University}

\author{\fnms{Enrique} \snm{Guerra Aguilar}\ead[label=e1]{eaguerra@mat.puc.cl}\thanksref{t1}}
\thankstext{t1}{Partially Supported by CAPES PNPD20130824, CONICYT FONDECYT Postdoctorado 3180255 and 
Nucleus Millenium Stochastic Models of Complex and Disordered Systems NC130062}
\address{\printead{e1}}
\affiliation{Pontificia Universidad Cat\'{o}lica de Chile}

\runauthor{Enrique Guerra}

\begin{abstract}
We prove a ballistic strong law of large numbers and an invariance
principle for random walks in strong mixing environments, under
condition $(T)$ of Sznitman (cf. \cite{Sz01}). This weakens for the first time
Kalikow's ballisticity assumption on mixing environments and proves the existence of arbitrary finite order moments for the approximate regeneration time of F. Comets and O. Zeitouni
\cite{CZ02}. The main technical tool in the proof is the introduction
of renormalization schemes, which had only been
considered for i.i.d. environments.
\end{abstract}



\end{frontmatter}

 \section{Introduction}
Random walk in a random environment (RWRE) is a well-known
stochastic model for random motion in random media, which
presents a wide range of applications going from DNA replication
models \cite{Ch62} up to for instance, a prototype for the study
of turbulent behavior in fluids \cite{Si82}. The model describes the
stochastic evolution of a particle on the lattice $\mathbb Z^d$,
where its transition probabilities are
in turn random. Within this framework, it is a fundamental and challenging question to
find the minimal local assumption that provides a given asymptotic behaviour for the walk.
For technical issues, the local assumption is usually strengthened to an assumption
of ballistic-type, the
target therefore is to prove a given behavior from one condition on the
environment and one ballisticity condition. In this work, assuming a mixing condition on the environment
and condition $(T)$ of Sznitman (cf. \cite{Sz01}-\cite{Sz02}), we shall prove ballistic regime
complemented with a diffusive scaling limit for the walk.

\vspace{0.5ex}
In the one-dimensional setting one can find almost complete
descriptions about RWRE asymptotic laws, scaling limits and
connections between different large scale concepts (see
\cite{Ze04}, Chapter 2 for a comprehensive review for $d=1$).
Throughout this article we focus on the higher dimensional case,
i.e. when the underlying dimension $d$ of the walk is greater than
$1$. A key role to prove our results will be played by
renormalization methods for mixing environments. The strategy of renormalization for RWRE
was introduced by Alain-Sol Sznitman in \cite{Sz00}, and further
developments can be found in subsequent articles as \cite{Sz01},
\cite{Sz02} and \cite{BDR14}, among others. In this article \textit{renormalization} for
RWRE is related to the theoretical construction of strategies that allow the walker to escape from traps (typically we are concerned with
traps which are slabs or large boxes) by the \textit{appropriate
boundary side}, with high probability. Overall, the
construction of these strategies involves the use of smaller traps
to be considered therein, which turns out a recursive procedure of
\textit{renormalization} nature. For i.i.d. random environments, estimates for exit probabilities from traps are established
with the help of the renewal structure of A-S. Sznitman and M. Zerner
\cite{SZ99}, a higher dimensional analogue of the
one previously introduced by H. Kesten in \cite{Ke77} for one-dimensional RWRE.

\vspace{0.5ex}
On the other hand, a kind of renewal structure for mixing
environments was introduced by F. Comets and O. Zeitouni in
\cite{CZ01}. This is an approximate renewal structure for general
mixing random environments. Indeed, the authors studied a quit weak
mixing assumption, the so-called \textit{cone mixing condition}.
They proved a law of large numbers for a class of strong
ballistic RWRE, where the hypotheses are: a strengthened form of
Kalikow's condition (cf. (\ref{Kalikow})), integrability conditions for the approximate regeneration time and the cone mixing assumption on the
environment (cf. \cite{CZ01}, Theorem 3.4). As the present work shows, the integrability conditions can be disposed provided we assume stronger mixing conditions on the environment. Alongside, a stronger mixing
condition on the environment has been investigated by F. Rassoul-Agha \cite{RA03},
which appears in the context of spin-glass systems at
high temperature and was introduced by R. Dobrushin and
S. Shlosman (cf. \cite{DS85}, see also \cite{Ma99} as a further
reference). Under Kalikow's condition, F. Rassoul-Agha proved a ballistic strong law of large numbers
by virtue of an appropriate extension of Kozlov's
theorem (see \cite{Ko85}). The approach to prove such extension appears when one sees
the stochastic evolution of the system from the point of view of
the particle. As a matter of fact, Rassoul-Agha's proof does not need to
assume a stronger version of Kalikow's condition, as was done in
the aforementioned result of \cite{CZ01}. However the point of view
of the particle relies on ergodic matters, making hard to visualize a proof
for the central limit theorem from this technique.

\vspace{0.5ex}
In this article we shall see that assuming condition $(T)$ along with renormalization
type of arguments, one has a Brownian scaling limit under the natural scaling of a ballistic walk. Indeed,
we shall
reconstruct or give meaning to part
of Sznitman's work \cite{Sz00}-\cite{Sz01}-\cite{Sz03} for i.i.d. environments, in a mixing setting. Thus the present article is fully
connected with the spirit of Feynman's phrase: ``There is pleasure
in recognising old things from a new viewpoint". As a result of
that recognition we will be able to weaken the ballisticity
assumption from Kalikow's to Sznitman's $(T)$ condition, proving
ballistic behaviour and a central limit theorem. Remarkably, we obtain analogously to the i.i.d. case the spirit
of a RWRE result: ballistic behaviour from one environment and one ballisticity assumptions. We also open a path for the investigation of ballistic behaviour under weaker assumptions than Kalikow's condition,
and we provide a partial answer to an open problem formulated in \cite{CZ01}
about the meaning of Sznitman
transient conditions in a mixing setting (cf. \cite{CZ01}, pp.
912-913, 6. Concluding remarks, item 3).

\vspace{0.5ex}
It is convenient at this point to fix some notation.
We only consider what is called in the RWRE literature as
a \textit{uniform elliptic random environment}, which means that
the walk has strictly uniform positive jump probabilities to each
nearest neighbour sites. More precisely, we pick an integer $d>1$
along with a positive real number $\kappa\in(0,1/(4d)]$ and denote
by $\mathcal{P}_{\kappa}$ the $2d-$dimensional simplex:
\begin{equation}\label{simplex}
  \mathcal{P}_{\kappa}:=\left\{z\in \mathbb R^{2d}: \Sigma_{1\leq i \leq 2d}\,z_i=1,\, z_i\geq 2\kappa \,\, \forall i \in[1,2d]\right\}.
\end{equation}
We consider the product space $\Omega=(\mathcal P_\kappa)^{\mathbb
Z^d}$ which is tacitly endowed with its canonical product
$\sigma$-algebra denoted by $\mathfrak{F}_{\Omega}$ and, for the
time being, fix a probability law $\mathbb P$ on
$\mathfrak{F}_{\Omega}$. Next, for a given random element $\omega
:=(\omega(y, e))_{\{y\in \mathbb Z^d, e\in \mathbb Z^d
:\,|e|=1\}}\in \Omega$, and $x\in \mathbb Z^d$, we define the
\textit{quenched law} $P_{x,\omega}$ as the law of the canonical
Markov chain $(X_n)_{n\geq0}$ with state space $\mathbb Z^d$ and
stationary transition probabilities satisfying
\begin{align*}
&P_{x,\omega}[X_0=x]=1,\\
&P_{x,\omega}[X_{n+1}=X_n+e|X_n]=\omega(X_n,e),\,\, |e|=1.
\end{align*}
One then defines the \textit{annealed law} $P_x$ of the random walk
via the semidirect product $\mathbb P\otimes P_{x,\omega}$ on the
product $\sigma-$algebra of the space $\Omega\times (\mathbb
Z^d)^{\mathbb N}$. It will be convenient to denote by $|\cdot|_1,
\, |\cdot|_2$ and $|\cdot|_{\infty}$, the $\ell_1, \, \ell_2$ and
$\ell_{\infty}$ norms, respectively. Furthermore, in this article
we will deal with distances between sets, and for instance
for $A, B \subset \mathbb Z^d$, the symbol $d_1(A, B)$ stands for
the $\ell_1$-distance between sets $A$ and $B$, i.e. $d_1(A, B):=\inf\{|x-y|_1,\, x\in A, y\in B\}$. Following X. Guo in
\cite{Gu14}, we now introduce the type of randomness on the environment of interest for us.
For this end, let us first recall the definition
of $r-$Markovian field.
\begin{definition}For $r>1$, let $\partial^r V=\{z\in \mathbb Z \setminus V: \exists y\in V, |z-y|_1\leq r\}$ be the $r$-boundary of the set $V \subset \mathbb Z$. A random environment $(\mathbb P, \mathfrak{F}_{\Omega})$ on $\mathbb Z^d$ is called $r$-Markovian if for any finite $V\subset \mathbb Z^d$,
\begin{equation*}
  \mathbb P[(\omega_{x})_{x\in V}\in \cdot|\mathfrak{F}_{V^c}]=\mathbb P[(\omega_x)_{x\in V}\in \cdot|\mathfrak{F}_{\partial ^r V}], \,\, \mathbb P-a.s.,
\end{equation*}
where $\mathfrak{F}_{\Lambda}=\sigma(\omega_x, \, x\in \Lambda)$.
\end{definition}
 Let $C$ and $g$ be positive real numbers. We will say that an $r$-Markovian field $(\mathbb P, \mathfrak{F}_{\Omega} )$ satisfies strong mixing condition \textbf{(SM)}$_{C,g}$ if for all finite subsets $\Delta\subset V \subset \mathbb Z^d$ with $d_1(\Delta, V^c)\geq r$, and $A\subset V^c$,
\begin{equation}
\label{sma}
\frac{d\mathbb P[(\omega_x )_{x\in \Delta}\in \cdot | \eta]}{d \mathbb P[(\omega_x )_{x\in \Delta}\in \cdot | \eta']}\leq \exp\left( C \sum_{x\in \partial^r \Delta, y \in \partial^r A}e^{-g|x-y|_1}\right)
\end{equation}
for $\mathbb P$-almost all pairs of configurations $\eta,
\,\eta'\in \Omega $ which agree over the
set $V^c \backslash A$. Here we have used the notation
\begin{equation*}
\mathbb P[(\omega_x )_{x\in \Delta}\in \cdot | \eta]=\mathbb P[(\omega_x )_{x\in \Delta}\in \cdot |\mathfrak{F}_{V^c}]|_{(\omega_x)_{x\in V^c}=\eta}.
\end{equation*}
We will also need a condition which is somehow weaker than the
previous one. We say an $r$-Markovian field $(\mathbb P,
\mathfrak{F}_{\Omega} )$ satisfies Guo's strong mixing condition \textbf{(SMG)}$_{C,g}$
if for all finite subsets $\Delta\subset V \subset \mathbb Z^d$
with $d_1(\Delta, V^c)\geq r$, and $A\subset V^c$,
\begin{equation}
\label{smg}
\frac{d\mathbb P[(\omega_x )_{x\in \Delta}\in \cdot | \eta]}{d \mathbb P[(\omega_x )_{x\in \Delta}\in \cdot | \eta']}\leq \exp\left( C \sum_{x\in \Delta, y \in  A}e^{-g |x-y|_1}\right)
\end{equation}
with the same notation as above.

\vspace{0.5ex}
Throughout this article, condition \textbf{(SM)}$_{C,g}$ will be
the main assumption on the environment and we will use condition
\textbf{(SMG)}$_{C,g}$ only with the purpose of using an asymptotic
\textit{more general} assumption. Strictly speaking,
\textbf{(SMG)}$_{C,g}$ is not implied by condition
\textbf{(SM)}$_{C,g}$, but in asymptotic terms it is harder to work
with \textbf{(SMG)}$_{C,g}$. The so-called Dobrushin-Sloshman condition implies \textbf{(SM)}$_{C,g}$,
for some constants $C$ and $g$ (cf. Lemma 9 of \cite{RA03}). We will not define Dobrushin-Sloshman condition and we refer to \cite{DS85}
for the original reference about this mixing assumption, and also to \cite{RA03} for a discussion more suitable for our purposes.

\vspace{0.5ex}
\noindent
We will now introduce condition \textbf{(T)}$_\ell$, where
$\ell$ is an element of the $d$-dimensional unit sphere $\mathbb S^{d-1}$ (cf.
\cite{Sz01}-\cite{Sz02}). As a result of Lemma \ref{T}, for $\ell\in
\mathbb S^{d-1}$ we can and do say that condition \textbf{(T)}$_l$
is satisfied, if there exists a neighborhood $U\subset \mathbb
S^{d-1}$ of $\ell$, so that for some $b, \tilde{b}>0$ one has
that
\begin{equation*}
\label{tgamma}
\limsup_{\substack{L\rightarrow\infty}} L^{-1}\log\left(P_0\left[\widetilde{T}_{-bL}^{l'}<T_{\tilde{b}L}^{l'}\right]\right)<0
\end{equation*}
holds, for all $l'\in U$, where we have used the standard notation: if $a\in \mathbb R$ and  $u\in \mathbb R\setminus\{0\}$, $T^u_a$ and $\widetilde{T}^u_{a}$ denote stopping times
defined as:
\begin{gather}
\label{rightst}
T^u_a:=\inf\{n\geq 0: X_n\cdot u\geq a\}\,\, \mbox{and}\,\,  \widetilde{T}^u_{a}:= T^{-u}_{-a}.
\end{gather}
We will point out that the \textit{exponential moment version} of
this condition (which is the original definition of \cite{Sz01},
page 726) does not make sense since we do not have
\textit{planar} regeneration times in mixing. Rather, we have
approximate \textit{cone} regeneration times (cf. Section
\ref{notatandprope}). The exponential moment and slab definitions
are equivalent for i.i.d environments (cf. \cite{Sz02}, Theorem 1.1).

\noindent
Our main result rests on a further assumption.
\begin{definition}
We say that assumption \textbf{(R)}$_{ g, \kappa}$ is
satisfied if:
\begin{equation}\label{Rgkappa}
g> 18\log\left(\frac{1}{\kappa}\right).
\end{equation}
\end{definition}
For i.i.d. environments one can take $g$ arbitrarily large in either: (\ref{sma}) or (\ref{smg}). On the other hand, one can construct \textit{non-degenerate} $r$-Markovian fields with properties (\ref{sma}) or (\ref{smg}) for any given intensity parameter $g>0$ (cf. \cite{Do94}-\cite{DS85}-\cite{Ma99}).

\vspace{0.2ex}
We obtain an annealed functional central limit for the natural scaling of a ballistic walk under the a priori
\textit{transient} \textbf{(T)}$_\ell$ condition.
\begin{theorem}
\label{mth3}
Let $C, g>0$ and $\ell\in \mathbb S^{d-1}$. Suppose that the RWRE
satisfies conditions \textbf{(T)}$_\ell$, either: \textbf{(SM)}$_{C,
g}$ or \textbf{(SMG)}$_{
C,g}$ and (\ref{Rgkappa}). Then
there exist a deterministic non-degenerate covariance matrix $R$
and a deterministic vector $v$ with $v\cdot \ell>0$, such that under
$P_0$; with $$ S_n(t):=\frac{X_{[nt]}-vt}{\sqrt{n}}, $$ the path
$S_n(t)$ taking values in the space of right continuous functions
possessing left limits equipped with the supremum norm, converges
in law to a standard Brownian motion with covariance matrix $R$.
\end{theorem}
\noindent
It is not our subject finite dependent environments, how-
ever let us mention that we can avoid the use of assumption (\ref{Rgkappa}) in that case. We refer to Remark \ref{remfdre} for a sketch of proof.

\noindent
Theorem \ref{mth3} is the first result in the direction of weakening
Kalikow's condition for a class of ballistic random walks in mixing environments. It is also for mixing environments the first time that an invariance principle is established from only one ballisticity condition. Denoting Kalikow's condition in
direction $\ell\in \mathbb S^{d-1}$ by \textbf{(K)}$_\ell$ (cf. (\ref{Kalikow})), we will prove
in Section \ref{seclln} the implication: \textbf{(K)}$_\ell$ $\rightarrow$ \textbf{(T)}$_\ell$. In general, the converse implication fails and we refer to
Section \ref{seclln} for further details.

\vspace{0.5ex}
We will now describe in some detail the contents and structure of
this article. Section \ref{notatandprope} gives equivalent formulations
for condition \textbf{(T)} and introduces the asymptotic renewal structure of Comets
and Zeitouni \cite{CZ01}. The random variable $\tau_1$ introduced there produces an \textit{almost}
regeneration property. The term almost is made precise in Section \ref{section2}, Proposition \ref{propare} and Corollary \ref{corren}. The crucial Section \ref{sectionexp} is mostly concerned with Proposition \ref{expmpr} and \ref{asymtrafluc}. These propositions show finiteness of some exponential moments for the random variable $|X_{\tau_1}|_2$ and a stretched exponential control on the probability of large fluctuation along the orthogonal space to the approximate asymptotic direction. Section \ref{sectailest} proves Theorem \ref{mth3} using the stretched exponential controls of Proposition \ref{asymtrafluc} together with renormalization to bound the tails of $\tau_1$. The last section will be devoted to prove that Kalikow's condition is stronger than \textbf{(T)}. We shall also see under Kalikow's condition that a strong law of large numbers of ballistic nature holds without the use of assumption (\ref{Rgkappa}), recovering by others methods F. Rassoul-Agha's theorem \cite{RA03} under a slightly weaker mixing hypothesis. Nevertheless, since the main assumption to construct the invariant measure $\mathbb {\widehat{P}}_{\infty}\ll \mathbb P$ in \cite{RA03} appeals to a ballistic estimate which is provided by Kalikow's condition (cf. (\ref{kball})) and the mixing condition is comparable to ours (cf. Lemma 7 in \cite{RA03}), it is possible that Rassoul-Agha's approach would apply under our assumptions as well.

\section{The Transient $(T)$ Condition and The Approximate Renewal Structure.}
\label{notatandprope}
We shall introduce the condition $(T)$ and recall the approximate regeneration time introduced in \cite{CZ01} by F. Comets and O. Zeitouni.

\subsection{On the $(T)$ Condition}
We begin with recalling the definition of directed system of slabs
as in \cite{Sz02}.
\begin{definition}
We say that $l_0,l_1, \ldots ,l_k\, \in \mathbb S^{d-1}$, $a_0=1,
a_1, \ldots, a_k\,>0$, $b_0, \ldots,b_k>0 $ generate an
$l_0$-directed systems of slabs of order $1$, when
\begin{itemize}
  \item $l_0, l_1\ldots, l_k$ generate $\mathbb R^d$
  \item $\mathcal D=\{x\in \mathbb R^d: x\cdot l_0\in [-b_0,1], l_i\cdot x\geq -b_i, i\in[1,k]\}\,\subset \{x\in \mathbb R^d: l_i\cdot x<a_i, \mathbf{\forall} i\in[1,k]\}$
  \item $\limsup_{\substack{M\rightarrow \infty}}M^{-1}\,\log P_0\left[\widetilde{T}_{-b_iM}^{l_i}<T_{a_iM}^{l_i}\right]<0$,
  for $i\in [0,k]$, with the convention $a_0=1$.
\end{itemize}
\end{definition}

For positives real numbers $L,\,L'$ and $l\in \mathbb S^{d-1} $, we introduce the box
$B_{L,L',l}(x)$ as
\begin{equation}
\label{generalboxes}
B_{L,L',l}(x):=
x+R\left(\left(-L,L\right)\times\left(-L',L'\right)^{d-1}\right)
\cap\mathbb{Z}^d,
\end{equation}
where $R$ is a rotation of $\mathbb R^d$ with $R(e_1)=l$ (the specific form of such
a rotation is immaterial for our purposes) and $x\in \mathbb Z^d$. For $V\subset \mathbb Z^d$, we set $\partial V=\partial^1 V$. Then for a given box $B_{L,L',l}(x)$ we define \textit{its positive boundary} $\partial ^+B_{L,L',l}(x)$ by
$$
\partial ^+B_{L,L',l}(x):= \partial B_{L,L',l}(x)\cap\left\{y\in \mathbb Z^d: (y-x)\cdot l\geq L\right\}.
$$
We also introduce for $A\subset \mathbb Z^d$ the exit time $T_A$ and the entrance time $H_A$ via:
\begin{align}
\nonumber
T_A&:=\inf\{n\geq 0:\, X_n\notin A\}\,\, \mbox{and }\\
\label{enandex}
H_A&:=\inf\{n\geq0:\, X_n\in A\}.
\end{align}
\medskip
\noindent
We can then prove
\begin{lemma}\label{T}
The following assertions are equivalents:
\begin{itemize}
\item [(i)]There exist data $l_0,l_1, \ldots ,l_k\, \in \mathbb S^{d-1}$, $a_0=1, a_1, a_2, \ldots, a_k\,>0$, $b_0,b_1, \ldots,$ $b_k>0$ generating an $l_0$-directed systems of slabs of order $1$.
\item [(ii)]For some positive constants $b$ and $\hat{r}$, and large $M$, there are finite subsets $\Delta_M\subset \mathbb Z^d$, with $0\in \Delta_M\subset \{x\in \mathbb Z^d: x\cdot l_0\geq -bM\}\cap \{x\in \mathbb R^d: |x|_2\leq \hat{r}M\}$ and
    $$
    \limsup_{\substack{M\rightarrow \infty}}M^{-1}\, \log P_0\left[X_{T_{\Delta_M}}\notin \partial^+\Delta_M\right]<0,
    $$
where $\partial^+\Delta_M=\partial \Delta \cap \{x\in \mathbb R^d:
x\cdot l_0\geq M \}$.
\item [(iii)]For some $\mathfrak{r}>0$, one has
$$
\limsup_{\substack{M\rightarrow\infty}}\, M^{-1}\log P_0\left[X_{T_{B_{M, \mathfrak{r}, l_0}(0)}}\notin \partial^+ B_{M, \mathfrak{r}M, l_0}(0)\right]<0.
$$
\end{itemize}
Furthermore, in case of any of them holds, we say that
\textbf{(T)}$_{l_0}$ (to be read as condition $T$ in direction
$l_0$) holds.
\end{lemma}
\begin{proof}
The proof of (i)$\Rightarrow$(ii) can be found in \cite{Sz02}, pp
516-517. Therefore, we turn to prove (ii)$\Rightarrow$(iii). By
(ii), there exist $b, \hat{r}>0$ , so that for large $M$ there are
finite subsets $\Delta_M$ with $0\in\Delta _M \subset \{x\in
\mathbb Z^d: x\cdot l_0\geq -bM\}\cap \{x\in \mathbb R^d: |x|_2\leq
\hat{r}M\}$ and
$$
    \limsup_{\substack{M\rightarrow \infty}}M^{-1}\, \log P_0\left[X_{T_{\Delta_M}}\notin \partial^+\Delta_M\right]<0.
$$
Therefore, one can find a constant $\widetilde{c}$ so that for
all large $M$:
$$
   P_0\left[X_{T_{\Delta_M}}\notin \partial^+\Delta_M\right]<e^{-\widetilde{c}M}.
$$
Furthermore, by taking the intersection of the set $\Delta_M$
with $\{x\in \mathbb Z^d: \,x\cdot l_0<M\}$, without loss of
generality we can and do assume that $\Delta_M\subset \{x\in
\mathbb Z^d: \, x\cdot l_0<M \}$. Consider the box
$\widetilde{B}_{M, \hat{r},b, l_0}(0)$ defined by
$$
\widetilde{B}_{M, r, b, l_0}(0)=\widetilde{R}\left((-bM,M)\times(-\hat{r}M, \hat{r}M)^{d-1}\right),
$$
where $\widetilde{R}$ is a rotation on $\mathbb R^d$ with
$\widetilde {R}(l_0)=e_1$. We have $\Delta_M \subset
\widetilde{B}_{M, \hat{r}, b, l_0}(0)$, and consequently for large
$M$,
\begin{equation}
\label{firstest}
 P_0\left[X_{T_{\widetilde{B}_{M, \hat{r}, b, l_0}(0)}}\in
\partial^+\widetilde{B}_{M, \hat{r}, b, l_0}(0)\right]\geq
P_0\left[X_{T_{\Delta_M}}\in \partial^+\Delta_M\right]>
1-e^{-\widetilde{c}M}.
\end{equation}
Notice that if $b\leq1$, we choose $\mathfrak{r}$ in
(iii) as $r$, and we finish the proof. Otherwise, we can proceed as
follows: we take $N=bM$ and consider the box $B_{N,
\hat{r}([b]+1)N, l_0}(0)$. We introduce for integer $i\in [1,[b]]$,
a sequence $(T_i)_{1 \leq i \leq [b]}$ of $(\mathcal F_n)_{n\geq0}-$stopping times via
\begin{align}
\nonumber
T_1&=T_{\widetilde{B}_{M, \hat{r}, b, l_0}(0)}, \,\, \mbox{ and for
$i>1$ }\\
\label{stoppinti}
T_i&=T_{\widetilde{B}_{M, \hat{r}, b,
l_0}(0)}\circ\theta_{T_{i-1}}+T_{i-1}.
\end{align}
As a result, we have
\begin{align}
\nonumber
  P_0\bigg[X_{T_{B_{N, \hat{r}([b]+1)N, l_0}(0)}}\in \partial^+B_{N, \hat{r}([b]+1)N, l_0}(0)\bigg]
  \geq & P_0\bigg[X_{T_{\widetilde{B}_{M, \hat{r}, b, l_0}(0)}}\in \\
  \label{caixas}
  \partial^+\widetilde{B}_{M, \hat{r}, b, l_0}(0),\ldots ,\big(X_{T_{\widetilde{B}_{M, \hat{r}, b, l_0}(0)}}\in  \partial^+\widetilde{B}_{M, \hat{r}, b, l_0}&(0)\big)\circ \theta_{T_{[b]}}\bigg].
\end{align}
It is convenient at this point to introduce \textit{boundary sets}
$F_i$, $i\in [1, [b]]$ as follows:
\begin{align*}
  F_1=&\partial^+ B_{M, \hat{r}, b, l_0}(0)\, \, \mbox{ and for $i>1$}\\
  F_i=&\bigcup_{y\in F_{i-1}}\partial^+B_{M, \hat{r}, b, l_0}(y),
\end{align*}
where $B_{M, \hat{r}, b, l_0}(y):=y+ B_{M, \hat{r}, b, l_0}(0)$. We
also introduce for $i\in[1,[b]]$, \textit{environment events} $G_i$
via
\begin{equation*}
  G_i=\left\{\omega\in \Omega:\, P_{y,\omega}\left[X_{T_{B_{M,\hat{r},b,l_0}(y)}}\in\partial^+B_{M,\hat{r},b,l_0}(y)\right]\geq 1 -e^{-\frac{\widetilde{c}}{2}M},\,\forall y\in F_i  \right\}.
\end{equation*}
Observe that the right-hand side of inequality (\ref{caixas}) is greater
than
\begin{align*}
   P_0&\left[X_{T_{\widetilde{B}_{M, \hat{r}, b, l_0}(0)}}\in \partial^+\widetilde{B}_{M, \hat{r}, b, l_0}(0),\left(X_{T_{\widetilde{B}_{M, \hat{r}, b, l_0}(0)}}\in \partial^+\widetilde{B}_{M, \hat{r}, b, l_0}(0)\right)\circ \theta_{T_1},\right.\\
  &\left. \ldots, \left(X_{T_{\widetilde{B}_{M, \hat{r}, b, l_0}(0)}}\in \partial^+\widetilde{B}_{M, \hat{r}, b, l_0}(0)\right)\circ \theta_{T_{[b]}}\mathds{1}_{G_{[b]}}\right]\\
  &=\sum_{y\in F_{[b]}}\mathbb E\left[P_{0,\omega}\left[X_{T_{\widetilde{B}_{M, \hat{r}, b, l_0}(0)}}\in \partial^+\widetilde{B}_{M, \hat{r}, b, l_0}(0),\ldots \right.\right.\\
  &\left.\left.\ldots, X_{T_{[b]}}=y\right]P_{y,\omega}\left[X_{T_{\widetilde{B}_{M, \hat{r}, b, l_0}(y)}}\in \partial^+\widetilde{B}_{M, \hat{r}, b, l_0}(y)\right]\mathds{1}_{G_{[b]}}\right]\\
  &\geq\left(1-e^{-\frac{\widetilde{c}}{2}M}\right)\left(P_0\left[X_{T_{\widetilde{B}_{M, \hat{r}, b, l_0}(0)}}\in \partial^+\widetilde{B}_{M, \hat{r}, b, l_0}(0),\ldots\right.\right.\\
  &\left.\left.\ldots, \left(X_{T_{\widetilde{B}_{M, \hat{r}, b, l_0}(0)}}\in \partial^+\widetilde{B}_{M, \hat{r}, b, l_0}(0)\right)\circ \theta_{T_{[b]-1}} \right]-\mathbb P[(G_{[b]})^c]\right)
\end{align*}
where we have made use of the Markov property. Iterate this argument
recursively to obtain:
\begin{align}\nonumber
P_0\big[X_{T_{B_{N, \hat{r}([b]+1)N, l_0}(0)}}\in & \partial^+B_{N,
\hat{r}([b]+1)N, l_0}(0)\big]\\
\label{finalcaixa}
&\geq\left(1-e^{-\frac{\widetilde{c}M}{2}}\right)^{[b]+1}-
\sum_{i=1}^{[b]}\left(1-e^{-\frac{\widetilde{c}M}{2}}\right)^{[b]-i}\mathbb P[(G_i)^c].
\end{align}
Notice that using (\ref{firstest}) along with Chebysev's inequality, we have for $i\in [1,[b]]$
and large $M$,
\begin{equation}\label{Gi}
\mathbb P[(G_i)^c]\leq\sum_{y\in F_i}\mathbb P\left[P_{y,\omega}\left[X_{T_{B_{M,\hat{r},b,l_0}(y)}}\notin \partial^+B_{M,\hat{r},b,l_0}(y)\right]>e^{-\frac{\widetilde{c}M}{2}}\right]\leq e^{-\frac{\widetilde{c}M}{4}}.
\end{equation}
From (\ref{finalcaixa}), the fact that $b$ is finite and
independent of $M$ and the estimate (\ref{Gi}); there exists a
constant $w>0$, so that for large $N$
\begin{equation*}
P_0\left[X_{T_{B_{N, \hat{r}([b]+1)N, l_0}(0)}}\in \partial^+B_{N,
\hat{r}([b]+1)N, l_0}(0)\right]\geq 1-e^{-wN}
\end{equation*}
and this ends the proof of the implication (ii) implies (iii) by taking $\mathfrak{r}=\hat{r}([b]+1)$.

\vspace{0.5ex}
\noindent
To prove the implication (iii)$\Rightarrow$(i), we fix a rotation
$R$ on $\mathbb R^d$, with $R(e_1)=l_0$ and such that $R$ is the
underlying rotation of hypothesis in (iii). For small $\alpha$ we
define $2(d-1)$-\textit{directions} $l_{+i}$ and $l_{-i}$, $i\in
[2,d]$
\begin{gather*}
  l_{+i}=\frac{l_0+\alpha R(e_i)}{|l_0+\alpha R(e_i)|_2} \,\, \mbox{ and }
  l_{-i}=\frac{l_0-\alpha R(e_i)}{|l_0-\alpha R(e_i)|_2}.
\end{gather*}
Following the same type of argument as in \cite{GR17}, Proposition
4.2, pp 13-15; but using exponential decay instead of polynomial
one; we conclude that there exists a small and positive $\alpha$,
so that for each $i\in [2,d]$ there are some $r_i>0$, with
\begin{equation}\label{piexp}
\limsup_{\substack{M\rightarrow}\infty}\,M^{-1} \log P_0\left[X_{T_{B_{M, r_i M, l_{\pm i}}(0)}}\notin \partial^+B_{M, r_i M, l_{\pm i}}(0)\right]<0,
\end{equation}
Thus, (\ref{piexp}) finishes the proof by taking
\begin{align*}
&a_0=1,\,a_1=a_2=\ldots=a_{2(d-1)}=\frac{1}{2},\\
&b_0=b_1=\ldots=b_{2(d-1)}=1 ,\\ &l_0,\,
l_1=l_{+1},l_2=l_{-1},\ldots, l_{2(d-1)-1}=l_{+(d-1)},
l_{2(d-1)}=l_{-(d-1)},
\end{align*}
and then observing that for integer $i\in[0,2(d-1)]$
\begin{equation*}
 P_0\left[\widetilde{T}_{-b_iM}^l<T_{a_iM}^l\right]\leq P_0\left[X_{T_{B_{M, r_i M,l_{i} }(0)}}\notin \partial^+B_{M, r_i M,l_{i} }(0)\right].
\end{equation*}
\end{proof}

\subsection{Approximate Renewal Structure}
\label{Apre}
Throughout this section we assume that condition \textbf{(T)}$_\ell$ holds, where $\ell\in \mathbb S^{d-1}$. We observe that one can and does assume $\ell$ so that there exists $h\in (0, \infty)$ with $h\ell=:l\in \mathbb Z^d$. This is not a further restriction since by item $i)$ of Lemma \ref{T}, the set $B\subset\mathbb S^{d-1}$ of directions $\ell\in B$ such that \textbf{(T)}$_\ell$ holds contains an open set, thus writing
\begin{equation*}
  A=\{u\in \mathbb S^{d-1}:\exists t \in (0,\infty) \,\,\mbox{with}\,\, tu\in \mathbb Z^{d}\}.
\end{equation*}
one has that $A$ is dense in $ \mathbb S^{d-1}$. Therefore we assume condition \textbf{(T)}$_\ell$, where $\ell$ is as above and choose a fixed $h>0$ with
\begin{equation}\label{rul}
l:=h\ell\in \mathbb Z^d.
\end{equation}
We will denote the canonical orthonormal basis by $e_i, \, i\in[1,d]$ and consider the probability
measure $\overline P_0$ given by
\begin{equation*}
\overline{P}_0:=\mathbb P\otimes Q\otimes P_{\omega, \varepsilon}^0\,\,\,\, \mbox{on}\,\,\, \Omega\times (\mathcal W)^{\mathbb N} \times (\mathbb Z^d)^{\mathbb N},
\end{equation*}
where $\mathcal{W}=\{z:\, z=\pm e_i, \, \mbox{ for some }\, i\in
[1,d]\}\cup\{0\}$, which is defined as follows: $Q$ is a product
probability measure such that with each sequence
$\varepsilon=(\varepsilon_1, \varepsilon_2, \,\ldots)\in (\mathcal
W)^{\mathbb N}$, for $i\in [1,d]$ we have $Q[\varepsilon_1=\pm
e_i]=\kappa$ and $Q[\varepsilon_1=0]=1-2d \kappa$. Then for fixed
random elements $\varepsilon\in (\mathcal W)^{\mathbb N}$ and
$\omega \in \Omega$, we define $P_{\omega, \varepsilon}^0$ as the
law of the Markov chain $(X_n)_{n\geq 0}$ with state space in
$\mathbb Z^d$, starting from $0\in \mathbb R^d$ and transition
probabilities
\begin{equation*}
  P_{\omega, \varepsilon}^0[X_{n+1}=X_n+e| X_n]=\mathds{1}_{\{\varepsilon_{n+1}=e\}}+\frac{\mathds{1}_{\{\varepsilon_{n+1}=0\}}}{1-2d\kappa}\left(\omega(X_n,e)-\kappa\right),
\end{equation*}
where $e$ is an element of the set $\{y\in \mathbb Z^d:\,
|y|_2=1\}$. The importance of this auxiliary probability space
stems from the easy to verify fact that the law of $(X_n)_{n\geq0}$
under $Q\otimes P_{\omega, \varepsilon}^0$ coincides with the law
under $P_{0,\omega}$, while the law under $\mathbb P \otimes
P_{\omega, \varepsilon}^0$ coincides with $P_0$.

\vspace{1ex}
\noindent
Define now the sequence $\bar{\varepsilon}$ of length $|l|_1$ in
the following way:
$\bar{\varepsilon}_1=\bar{\varepsilon}_2=\ldots=\bar{\varepsilon}_{|l_1|}=\mbox{sign}(l_1)e_1$,
$\bar{\varepsilon}_{|l_1|+1}=\bar{\varepsilon}_{|l_1|+2}=\ldots=\bar{\varepsilon}_{|l_1|+|l_2|}=\mbox{sign}(l_2)e_2$,
$\ldots , \,
\bar{\varepsilon}_{|l|_1-|l_d|+1}=\ldots=\bar{\varepsilon}_{|l|_1}=\mbox{sign}(l_d)e_d$.
Define for $\zeta>0$ small, $x\in \mathbb Z^d$, the cone
$C(x,l,\zeta)$ by
\begin{equation}\label{cone}
C(x,l,\zeta):=\{y\in \mathbb Z^d: (y-x)\cdot l \geq \zeta |l|_2
|y-x|_2 \}.
\end{equation}
We will assume that $\zeta$ is small enough in order to satisfy the
following requirement:
\begin{equation*}
  \bar{\varepsilon}_1,\bar{\varepsilon}_1+\bar{\varepsilon}_2,\, \ldots\, , \bar{\varepsilon}_1+\bar{\varepsilon}_2+\ldots+\bar{\varepsilon}_{|l|_1}\in C(0, l, \zeta).
\end{equation*}
For $L\in |l|_1\mathbb N$ we will denote by
$\bar{\varepsilon}^{(L)}$ the vector
\begin{equation*}
\bar{\varepsilon}^{(L)}=\overbrace{(\bar{\varepsilon}, \bar{\varepsilon},\, \ldots\, ,\bar{\varepsilon},\bar{\varepsilon})}^{L/|l|_1-\mbox{times}}
\end{equation*}
of length equal to $L$. Setting
\begin{equation*}
  D':=\inf\{n\geq0:\, X_n\notin C(X_0, l , \zeta)\},
\end{equation*}
we have:

\begin{lemma}\label{Dundert}
Assume condition \textbf{(T)}$_\ell$, and fix $\mathfrak{r}$ and a rotation $R$ as in item $iii)$ of
Lemma \ref{T}. Then there exists $c_1>0$ such that if $\zeta
<\min\left\{\frac{1}{9d}, \frac{1}{3d\mathfrak{r}}\right\}$, then
$$
P_0[D'=\infty]\geq c_1.
$$
\end{lemma}
\begin{proof}
For $x\in \mathbb Z^d$ and $\alpha>0$, we define the \textit{flat
cone} $\mathcal{C}(x, \alpha, \ell)$ by
\begin{align}\label{flatcone}
\mathcal{C}(x, \alpha, \ell)=&\left\{y\in \mathbb Z^d:\, (y-x)\cdot\frac{\ell+\alpha R(e_i)}{|\ell+\alpha R(e_i)|_2}\geq 0,\right.\\
\nonumber
 & \left. (y-x)\cdot\frac{\ell-\alpha R(e_i)}{|\ell-\alpha R(e_i)|_2}\geq 0, \,\forall i\in[2,d] \right\}.
\end{align}
It is clear when $y\in \mathcal C(x,\alpha,\ell)$, using the
fact that for $i\in [2,d]$, $|\ell\pm \alpha R(e_i)|_2>0$ (since
$R(e_1)=\ell$), if $\alpha<1$ one has for all $i\in [2,d]$:
\begin{align*}
 (y-x)\cdot l\geq& \alpha|(y-x)\cdot R(e_i)|\\
 (y-x)\cdot l\geq&  \frac{\alpha}{d}  \sum_{i=1}^d|(y-x)\cdot R(e_i)|\geq \frac{\alpha}{d} |y-x|_2.
\end{align*}
As a result $\mathcal{C}(x, \alpha, \ell)\subset C(x, \ell,
\frac{\alpha}{d})=C(x,l, \frac{\alpha}{d})$. On the other hand, the
polynomial condition $(WP)$ of \cite{GR17} page 11, is obviously
implied by $iii)$ of Lemma \ref{T}. We finish the proof by applying
Proposition 5.1 of \cite{GR17}.
\end{proof}

We choose $\zeta>0$ satisfying the hypotheses of Lemma \ref{Dundert}. For each $L\in |l|_1\mathbb N$,
we define $S_0=0$, and denoting by $\theta$ the canonical time
shift, we set
\begin{align*}
S_1&=\inf\left\{n\geq L:\, X_{n-L}\cdot l >\max_{0\leq j<n-L}\{X_j\cdot
l\}, \,
(\varepsilon_{n-L},\ldots,\varepsilon_{n-1})=\bar{\varepsilon}^{(L)}\right\},\\
R_1&=D'\circ \theta_{S_1}+S_1,
\end{align*}
and for $n>1$
\begin{align*}
S_n&\\
=&\inf\left\{n>R_{n-1}:\, X_{n-L}\cdot l >\max_{0\leq j<n-L}\{X_j\cdot
l\}, \,
(\varepsilon_{n-L},\ldots,\varepsilon_{n-1})=\bar{\varepsilon}^{(L)}\right\},\\
R_n&=D'\circ \theta_{S_n}+S_n,
\end{align*}
where we define $S_n=\infty$ or $R_n=\infty$ whenever the respective previous random variable is $\infty$.
For given $L$ as above, these random variables are stopping times for the canonical underlying filtration of the pair $(X_n,
\varepsilon_n)_{n\geq 0}$. Notice also that the chain of inequalities
\begin{equation*}
  S_0=0<S_1\leq R_1\leq \ldots \leq S_n\leq R_n \ldots \leq \infty
\end{equation*}
is satisfied, with strict inequality if the left member is finite. Indeed, we shall see in brief that under assumption \textbf{(T)}$_{\ell}$ all of them are strict inequalities. Setting
\begin{equation*}
  K:=\inf\{n\geq 1: S_n<\infty, R_n=\infty\},
\end{equation*}
one defines the first time of
asymptotic regeneration $\tau_1:=S_K \leq \infty$ (we shall drop $L$ from the notation when there is not risk of confusion). A qualitative characterization
of the time $\tau_1=n$ is as follows: the first time $n$ that the walk
takes a strict record level in direction $l$ at time $n-L$, after which the
walk is pushing through direction $l$ by unit steps on the
lattice $\mathbb Z^d$ just owed to the action of
$\bar{\varepsilon}^{(L)}$ sequence in the probability space $(Q,
(\mathcal W)^{\mathbb N})$, independently on the environment, and finally for any future $j>n$ the walk remains forever inside the cone
$C(X_n, l, \zeta)$.

\vspace{0.5ex}
The next lemma shows that the previous construction is significant and its proof can be derived from Lemma \ref{Dundert} in conjunction with the argument given in \cite{Sz02}, page 517.
\begin{lemma}\label{tranunderttau}
Assume \textbf{(T)}$_\ell$. Then $P_0$-a.s. (see (\ref{rul}))
\begin{equation}\label{TTRANSIENT}
\lim_{\substack n\rightarrow \infty}\, X_n\cdot l=\infty.
\end{equation}
and there exists a deterministic $L_0>0$, so that
for each $L\geq L_0$, with $L\in \,|l|_1\mathbb N$, one has $\overline{P}_0$-a.s.
\begin{equation}\label{tau1finite}
\tau_1^{(L)}<\infty.
\end{equation}
\end{lemma}
Choosing $L$ and $\zeta$ as prescribed by Lemmas \ref{Dundert}-\ref{tranunderttau}, one has that $\overline{P}_0$-a.s. $\{R_k<\infty\}\,=\,\{S_{k+1}<\infty\}$ and $S_1<\infty$ by (\ref{TTRANSIENT}).

\noindent
Let us now define the iterates regeneration times of $\tau_1$ via:
\begin{equation*}
  \tau_n=\tau_1\circ \theta_{\tau_{n-1}}+\tau_{n-1}
\end{equation*}
for $n>1$. It is easy to verify that for any $k\in \mathbb N$, $\overline P_0-$a.s. $\tau_k<\infty$.

The main technical objective of this article will be to obtain
upper bounds for the $L$ dependent probabilities
\begin{equation*}
\overline{P}_0[\tau_1>u],
\end{equation*}
where $u$ is large and independent on a fixed $L$.

\subsubsection{General Proof Strategy}
From the fact that the proof of our main result Theorem \ref{mth3}
is a bit involved, we shall explain the general strategy to
follow. Roughly speaking, we will try to recover all of the
Sznitman's results of \cite{Sz00} to bound the probability of the
asymptotic regeneration time tails and then applying a version of the central limit theorem in \cite{CZ02} we
will obtain the proof. However, extending these results to the strong mixing case will prove to be technically more challenging.

\section{On the Almost Renewal Structure for Random Walks in Strong Mixing Environments}
\label{section2}
Our mixing assumptions provide an approximate
renewal structure when one considers the increments of the $\tau_1$
iterates. More precisely, we let $x\mathbb Z^d$ and $L\in|l|_1 \mathbb N$  and define the $\sigma$-algebra:
\begin{equation*}
\mathcal G_1:=\sigma\left(\omega(y, \cdot): y\cdot l< X_{\tau_1}\cdot l- \frac{L|l|_2}{|l|_1},\, \left(\varepsilon_i\right)_{0\leq i\leq \tau_1},\,\left(X_i\right)_{0 \leq i\leq \tau_1}\right),
\end{equation*}
along with the random environment $\sigma-$algebra
\begin{equation}\label{sigmafrak}
 \mathfrak{F}_{x,L}:=\sigma\left(\omega(y,\cdot): (y-x)\cdot l\leq -\frac{L|l|_2}{|l|_1} \right).
\end{equation}
An important technical fact comes in the next
\begin{proposition}[Under either: \textbf{(SM)}$_{C, g}$ or \textbf{(SMG)}$_{C, g}$ ]\label{propare}
For $L\in|l|_1 \mathbb N$ we let $\mu:=\mu(L)=\exp\left(e^{-gtL}\right)$. Then
for each $t\in (0,1)$ there exists $L_0=L_0(C, g, \kappa, l,
d,r)\in |l|_1 \mathbb N$ such that $\overline{P}_0-$a.s.
\begin{align}
\nonumber
\mu^{-1}(L)\,\overline{P}_0[(X_{n}-X_{0})_{n\geq0}&\in \cdot\,|\, D'=\infty] \leq
\overline{P}_0[(X_{\tau_1+n}-X_{\tau_1})_{n\geq0}\in \cdot \,|\,\mathcal G_1]\\
\label{approxre1}
\leq& \mu(L)\,\overline{P}_0[(X_{n}-X_{0})_{n\geq0}\in \cdot\,|\, D'=\infty]
\end{align}
holds, for all $L\geq L_0$, $L\in|l|_1 \mathbb N$.
\end{proposition}
\begin{proof}
We fix $t$ as in the statement of the proposition and consider
non-negative bounded functions $f$ and $h$ which are
$\sigma((X_n)_{n\geq0})$ and $\mathcal{G}_1$ measurable,
respectively. Denoting by $\vartheta$ and $\theta$ the space and
time shifts, from the very definition of the renewal structure one
has,
\begin{align*}
\overline{E}_0[f(X_{\tau_1+\cdot}-X_{\tau_1})h]=\sum_{k\geq 1}\overline{E}_0&[f(X_{S_k+\cdot}-X_{S_k})h,
S_k<\infty, R_k=\infty]\\
=\sum_{k\geq 1, j\geq 1, x\in \mathbb
Z^d}\overline{E}_0[f(X_{S_k+\cdot}-x)h,& X_{S_k}=x, S_k=j,
D'\circ\theta_n=\infty].
\end{align*}
Observe that over the event $\{X_{S_k}=x, S_k=j,
D'\circ\theta_j=\infty\}$ one can find a bounded function
$h_{x,k,j}$, which is $\sigma((\omega(y,\cdot), y\cdot l< x\cdot l
-(L|l|_2)/(|l|_1))\otimes (X_n)_{0\leq n\leq j})$-measurable and
equal to $h$. As a result, the rightmost term in the previous
display equals
\begin{equation*}
\sum_{k,j\geq 1, x\in \mathbb Z^d}\mathbb E[ E_{Q\otimes P_{\varepsilon, \omega}^0}[f(X_{S_k+\cdot}-x)h_{x,k,j}, X_{S_k}=x, S_k=j, D'\circ\theta_n=\infty]]
\end{equation*}
Applying now the strong Markov property at time $S_k$ and using the
product structure of $Q$ one sees in turn that equals
\begin{align}
\nonumber
\sum_{k,j\geq 1, x\in \mathbb Z^d}&\mathbb E \bigg[E_{Q\otimes P_{\varepsilon, \omega}^0}[h_{x,k,j}, X_{S_k}=x, S_k=j]\\
\label{re1}
\times &E_{Q\otimes P_{\vartheta_n \varepsilon, \theta_x\omega}^0}[f(X_{\cdot}-x), D'=\infty]\bigg].
\end{align}
Use notation (\ref{sigmafrak}) to obtain that (\ref{re1}) equals
\begin{align}
\nonumber
\sum_{k,j\geq 1, x\in \mathbb Z^d}&\mathbb E \Big[E_{Q\otimes P_{\varepsilon, \omega}^0}[h_{x,k,j}, X_{S_k}=x, S_k=j]\\
\label{reimpor}
\times & \mathbb E[E_{Q\otimes P_{\vartheta_n \varepsilon, \theta_x\omega}^0}[f(X_{\cdot}-X_0), D'=\infty]| \mathfrak{F}_{x,L}] \Big].
\end{align}
Fix $x\in \mathbb Z^d$, $n\in \mathbb N$ and consider the
conditional probability distribution
$$
\hat{\mathbb P}[\cdot\,|\mathfrak{F}_{x,L}]:=\frac{\mathbb E[ P_{Q\otimes P_{\vartheta_n \varepsilon, \theta_x\omega}^0}[(X_{i}-X_0)_{i\geq 0}\in \cdot, D'=\infty]|\mathfrak{F}_{x,L}]}{\mathbb E[ P_{Q\otimes P_{\vartheta_n \varepsilon, \theta_x\omega}^0}[D'=\infty]|\mathfrak{F}_{x,L}]}.
$$
It will be proven below that there exists a positive constant $L_0>0$, so that for each $L\in
|l|_1\mathbb N$, $L\geq L_0$ we have $\overline{P}_0$-a.s.
\begin{align}
\nonumber
&\exp\left(-e^{-g\, tL}\right)\overline{P}_0[(X_{i}-X_0)_{i\geq 0}\in \cdot\,|\, D'=\infty]\leq \hat{\mathbb P}[\cdot\,|\mathfrak{F}_{x,L}]\\
\label{acinequality}
&\leq \exp\left(e^{-g\, tL}\right)\overline{P}_0[(X_{i}-X_0)_{i\geq 0}\in \cdot\,|\, D'=\infty].
\end{align}
Thus using (\ref{acinequality}) and (\ref{tau1finite}), writing (\ref{reimpor}) as
\begin{align*}
A=\sum_{k,j\geq 1, x\in \mathbb Z^d} \mathbb E \Bigg[&E_{Q\otimes
P_{\varepsilon, \omega}^0}[h_{x,k,j}, X_{S_k}=x, S_k=j]\\
\times\mathbb E[ P_{Q\otimes P_{\vartheta_n \varepsilon, \theta_x\omega}^0}[D'=\infty]|\mathfrak{F}_{x,L}]
&\frac{\mathbb E [E_{Q\otimes P_{\vartheta_n \varepsilon, \theta_x\omega}^0}[f(X_{\cdot}-X_0), D'=\infty]| \mathfrak{F}_{x,L}]}{\mathbb E[ P_{Q\otimes P_{\vartheta_n \varepsilon, \theta_x\omega}^0}[D'=\infty]|\mathfrak{F}_{x,L}]}\Bigg]
\end{align*}
one has
\begin{equation*}
\exp\left(-e^{-g\, tL}\right)\overline{E}_0[h]\overline{E}_0[f| D'=\infty]\leq A\leq \exp\left(e^{-g\, tL}\right)\overline{E}_0[h]\overline{E}_0[f| D'=\infty]
\end{equation*}
which finishes the proof.
\end{proof}
Let us now prove the claim (\ref{acinequality}). Our proof shares some
similarities with the proofs of X. Guo in
Lemma 5 and Proposition 7 of \cite{Gu14}.
\begin{lemma}
\label{lemmaac}
Under the assumptions and notation of Proposition \ref{propare}.
Let $x_0\in \mathbb Z^d$ and $n\in \mathbb N$, then there exists
$L_0=L_0(C,g,\kappa, l, d , r)\in |l|_1 \mathbb N$ such that
\begin{align*}
&\exp\left(-e^{-g\, tL}\right)\overline{P}_0[(X_{i}-X_0)_{i\geq 0}\in \cdot\,|\, D'=\infty]\leq \hat{\mathbb P}[\cdot\,|\mathfrak{F}_{x_0,L}]\\
&\leq \exp\left(e^{-g\, tL}\right)\overline{P}_0[(X_{i}-X_0)_{i\geq 0}\in \cdot\,|\, D'=\infty],
\end{align*}
for all $L\geq L_0$, with $L\in |l|_1\,\mathbb N$.
\end{lemma}
\begin{proof}
We split the proof into three steps.
\begin{description}
   \item[\textit{Step 1.}] The first step is the following claim:

\noindent
Let $A\subset\Lambda\subset\mathbb Z^d$. Suppose $S\neq\varnothing$
is a countable set of finite paths $x_.=(x_i)_{i=0}^N,\, N<\infty$
starting at $x_0$ that satisfy $d_1(x., \Lambda)\geq r$ and
\begin{equation}
\label{sum1}
  \sum_{y\in A, 0\leq i \leq N}e^{-g|y-x_i|_1 }\leq a,
\end{equation}
uniformly on $N$. Then $\mathbb P$-a.s. (cf. \cite{Gu14}, page 381 for a proof)

\begin{align*}
&\exp\left(-Ca\right)\\
&\leq \frac{\mathbb E[E_Q[P_{ \omega \circ \theta_{x_0}, \varepsilon \circ \vartheta_n}[\bigcup_{N\geq 0}\{(X_i-X_0)_{0\leq i\leq N}\in S\}]]|\omega_y,y\in\Lambda]}{\mathbb E[E_Q[P_{ \omega \circ \theta_{x_0}, \varepsilon \circ \vartheta_n}[\bigcup_{N\geq 0}\{(X_i-X_0)_{0\leq i\leq N}\in S\}]]|\omega_y,y\in\Lambda \backslash A]}\\
&\leq \exp\left(C a\right).
\end{align*}
   \item[\textit{Step 2.}]
   Consider the hyperplane $H_{L,l}$ defined by
   $$
   H_{L,l}:=\{z\in \mathbb Z^d: z\cdot l\leq -(L|l|_2)/|l|_1\}.
   $$
   In this step, we will first estimate the series
   \begin{align}
   \label{series1}
   &\sum_{\substack{y\in \partial ^r H_{L,l},\\ z\in \partial^r C(0,l,\zeta)}}\exp\left(-g |y-z|_1\right)\hspace{1.5ex}\mbox{ and }\\
   \label{series2}
   &\sum_{\substack{y\in  H_{L,l},\\ z\in  C(0,l,\zeta)}}\exp\left(-g |y-z|_1\right)
   \end{align}
   in terms of $g$, for some large but fixed $L$. Notice that for given $L>0$, both series in (\ref{series1}) converge because $\zeta>0$, as follows from the next argument. Choose $\hat{t}\in(t,1)$ and consider the first series in (\ref{series1}). We take $L$ large enough so that $L>(1-\hat{t})^{-1}2r$ (thus $L-2r>\hat{t}L$) and applying condition \textbf{(SM)}$_{C,g}$.
   $$
   \sum_{n\geq 0}\sum_{(y,z) \in \mathcal{HC}_{L,n, y, z} }e^{-g|y-z|_1} ,
   $$
   where we have written
  \begin{align*}
  &\mathcal{HC}_{L, n, y, z}:=\{(y,z) :y\in \partial ^r H_{L,l},\, z\in \partial ^r C(0,l,\zeta), \\
  &\hat{t}L + n \leq |y-z|_1 \leq \hat{t}L + (n+1)\}.
  \end{align*}
     Above was used the fact that the minimal $|\,\cdot\,|_1$-distance between any two points $y\in \partial ^r H_{L,l},\, z\in \partial ^r C(0,l,\zeta)$ is at least $L-2r$.

  \noindent
  Therefore we obtain the following upper bound for series (\ref{series1}):
  \begin{equation*}
  \sum_{n\geq0}|H_{L, n, y, z}|e^{-g (\hat{t}L+n)}.
  \end{equation*}
  On the other hand, the estimate
  $$
  |H_{L, n, y, z}| \leq \tilde{c}(d)r^2(n+1)^{2(d-1)}
  $$
  holds, for a suitable $\tilde{c}>0$ depending on $d$ and $\zeta$. Notice also that
  $$
  \sum_{n\geq 0}(n+1)^{2(d-1)}\, e^{-g n}
  $$
  converges, thus combining both last estimates we conclude: there exists $C_1=C_1(C, d, g, r, \zeta, l)>0$ scuh that if $L\geq C_1$ one can bound from above series (\ref{series1}) by
  $$
  \exp\left(-g \,\widetilde{t}L\right),
  $$
  where $\widetilde{t}\in (t, \hat{t})$.

  \noindent
  Performing the same type of argument, one sees that from the fact that the inner angle of the cone is positive there exists $C_2>0$, so that:
  \begin{equation}
   \label{series3}
   \sum_{y\in  H_{L,l}, z\in  C(0,l,\zeta)}\exp\left(-g |y-z|_1\right)\leq \exp\left(-g\, tL\right)
   \end{equation}
  holds, for all $L\in \mathbb N |l|_1$, $L\geq L_0$, provided that $L_0\geq C_2$.

  Consequently, for a given finite path starting from $x_0$ of the form
  $$
  x_.=(X_i)_{i=0}^N, \, N<\infty, \,x_.\subset C(x_0, l, \zeta)
  $$
  one has that uniformly on $N$, there exists a positive constant $C_3$ such that if $L\geq C_3$
  $$
  \sum_{y\in \partial ^r H_{L,l,x_0}, z\in \partial^r G_x}\exp\left(-g|y-(z-x_0)|_1\right)\leq e^{-g\, \widetilde{t}L},
  $$
  provided that we define
  $$
  H_{L,l,x_0}:=\{z\in \mathbb Z^d: (z-x_0)\cdot l\leq -(L|l|_2/(|l|_1))\}
  $$
  and
  $$
  G_x:=\{y\in \mathbb Z^d: y =X_i, \mbox{ for some }i\in[0, N]\}.
  $$
  Likewise using the second estimate in (\ref{series3}), we obtain a suitable constant $C_4$ such that $L_0\geq C_4$ implies that
  $$
  \sum_{y\in H_{L,l,x_0}, 0\leq i\leq N}\exp\left(-g|y-(X_i-x_0)|_1\right)\leq e^{-g\, \widetilde{t}L}
  $$
  holds, for $L\geq L_0$, uniformly on $N\in \mathbb N$, where the notation is as above.

  \noindent
  We then consider, instead of a fixed path $x_\cdot$, a countable collection $S$ of finite paths starting from a common point $x_0\in\mathbb Z^d$ with all of them contained in a cone $C(x_0,l, \zeta)$. Therefore, choosing $\widehat{t}\in(t,\widetilde{t})$ we find that there exists $C_5$ so that whenever $L\geq C_5$, \textbf{\textit{Step 1}} gives
  \begin{align*}
  &\exp\left(-e^{-g\, \widehat{t}L}\right) \\
  &\leq\frac{\mathbb E[E_Q[P_{ \omega \circ \theta_{x_0}, \varepsilon \circ \vartheta_n}[\bigcup_{N\geq 0}\{(X_i-X_0)_{0\leq i\leq N}\in S\}]]|\omega_y,y\in\Lambda]}{\mathbb E[E_Q[P_{ \omega \circ \theta_{x_0}, \varepsilon \circ \vartheta_n}[\bigcup_{N\geq0}\{(X_i-X_0)_{0\leq i\leq N}\in S\}]]|\omega_y,y\in\Lambda \backslash A]} \\
  &\leq \exp\left(e^{-g\, \widehat{t}L}\right),
  \end{align*}
  where $\Lambda=H_{L,l,x_0}$, and $A$ is an arbitrary subset of $\Lambda$.
   \item[\textit{Step 3.}] We prove here the assertion of the lemma. For $j\in \mathbb N$, we set $S_{0,j}$ the set of paths of length $j-1$ starting from $0$. Then by definition one has
   \begin{align*}
   &\{(X_{i}-X_0)_{i\geq 0}\in \cdot\,\,, D'=\infty\}\\
   &=\bigcap_{n\geq 0}\bigcup_{N\geq 0}\bigcup_{j=0}^N\{(X_i-X_0)_{i=0}^{j}\in S_{0,j},D'>n \}.
   \end{align*}
   For any $n\in \mathbb N$, an application of \textit{Step 1} and \textit{Step 2} lead us to
   \begin{align*}
   &\exp\left(-e^{-g\, \widehat{t}L}\right)\\
   &\leq\frac{\mathbb E[E_Q[P_{ \omega \circ \theta_{x_0}, \varepsilon \circ \vartheta_n}[(X_i-X_0)_{i\geq 0}\in \cdot, D'>n]]|\omega_y,y\in\Lambda]}{\mathbb E[E_Q[P_{ \omega \circ \theta_{x_0}, \varepsilon \circ \vartheta_n}[(X_i-X_0)_{i\geq 0}\in \cdot, D'>n]]|\omega_y,y\in\Lambda \backslash A]} \\
   &\leq \exp\left(e^{-g\, \widehat{t}L}\right),
  \end{align*}
  where $\Lambda$ and $A$ are as in \textbf{\textit{Step 2}} (recall that $A$ is an arbitrary subset of $\Lambda$). Letting $n\rightarrow\infty$ and then using the result for $A=\Lambda$, one gets
  \begin{align*}
  &\exp\left(-e^{-g\, \widehat{t}L}\right)\\
   &\leq \frac{\mathbb E[E_Q[P_{ \omega \circ \theta_{x_0}, \varepsilon \circ \vartheta_n}[(X_i-X_0)_{i\geq 0}\in \cdot, D'=\infty]]|\omega_y,y\in\Lambda]}{P_0[(X_i-X_0)_{i\geq 0}\in \cdot, D'=\infty]} \\
  &\leq \exp\left(e^{-g\, \widehat{t}L}\right),
  \end{align*}
  and
  \begin{align*}
  &\exp\left(-e^{-g\, \widehat{t} L}\right)\\
  &\leq\frac{\mathbb E[E_Q[P_{ \omega \circ \theta_{x_0}, \varepsilon \circ \vartheta_n}[ D'=\infty]]|\omega_y,y\in\Lambda]}{P_0[ D'=\infty]}
  \leq \exp\left(e^{-g\, \widehat{t}L}\right).
  \end{align*}
  By choosing $L_0$ large enough such that for $L\geq L_0$
  $$
  2e^{-g\, \widehat{t}L}\leq e^{-g\,tL},
  $$
  we finish the proof.
 \end{description}
\end{proof}
We close this section with a straightforward
consequence of the previous proposition which will be stated in the
next corollary, for reference purposes. As a natural extension to
$\mathcal G_1$, we define the sigma-algebra $\mathcal G_i$, where
$i\in \mathbb N$, by
\begin{equation*}
  \mathcal G_i=\sigma\big(\omega(y, \cdot): y\cdot l< X_{\tau_i}\cdot l- (L|l|_2)/(|l|_1),\, (\varepsilon_i)_{0\leq j\leq \tau_i},\,(X_j)_{0 \leq j\leq \tau_i}\big).
\end{equation*}
Let $\mu$ be as in the statement of Proposition \ref{propare}, then an induction argument makes us conclude:
\begin{corollary}
\label{corren}
Assume either: \textbf{(SM)}$_{C, g}$ or \textbf{(SMG)}$_{C, g}$ and let
$j\in\mathbb N, \, t\in (0,1)$. Then there exists $L_0=L_0(C,g,
\kappa, l, d,r)\in |l|_1 \mathbb N$ such that $\mathbb{P}$-a.s.
\begin{align*}
&\mu^{-1}(L)\,\overline{P}_0[(X_{n}-X_{0})_{n\geq0}\in \cdot\,|\, D'=\infty] \leq \overline{P}_0[(X_{\tau_j+n}-X_{\tau_j})_{n\geq0}\in \cdot \,|\,\mathcal G_j]\\
&\leq \mu(L) \overline{P}_0[(X_{n}-X_{0})_{n\geq0}\in \cdot\,|\, D'=\infty]
\end{align*}
holds, for all $L\geq L_0$ with $L\in |l|_1\mathbb N$.
\end{corollary}

\section{Preliminary Estimates: The Regeneration Position has some Exponential Moments}
\label{sectionexp}

It is the purpose of this section to prove that the random variable
$X_{\tau_1}\cdot l$ has some finite exponential moments under condition
\textbf{(T)}$_\ell$ (recall (\ref{rul})). We will derive after that proof two further consequences. On the one hand it will be showed the finiteness of some exponential moments for the random variable $\sup_{\substack{0\leq n\leq \tau_1}}|X_n|_2$; and on the other hand, an upper bound of stretched
exponential-type for the probability of large orthogonal
oscillations along the approximate asymptotic direction of the walk. Throughout the rest of the paper we assume condition \textbf{(T)}$_\ell$, and we pick $h\in (0,\infty)$ so that (\ref{rul}) is satisfied. Then we choose a constant $\mathfrak{r}>0$ as in the item $iii)$ of Lemma \ref{T} and the cone angle
$\zeta$ will be any fixed positive number satisfying the following
requirement
\begin{equation}
\label{zeta}
\zeta<\min\left\{\frac{1}{9d},\, \frac{1}{3d\mathfrak{r}},\,\cos\left(\frac{\pi}{2}-\arctan(3\mathfrak{r})\right)\right\}.
\end{equation}

\begin{proposition}
\label{expmpr}
Assume that \textbf{(T)}$_\ell$ and either \textbf{(SM)}$_{C,g}$ or \textbf{(SMG)}$_{C,g}$ hold. Then
there exist positive constants $c_2$, $c_3$ and $L_0$, such that
for all $L\geq L_0$, with $L\in |l|_1 \mathbb N$:

\begin{equation}
\label{expmompos}
\overline{E}_0[\exp\left(c_2\kappa^L X_{\tau_1}\cdot l\right)]<c_3
\end{equation}
holds.

\end{proposition}

\begin{proof}
By virtue of the renewal structure definitions, for $c>0$ and $L\in
|l|_1 \mathbb{N}$, one has that:
\begin{align*}
\overline{E}_0\left[\exp\left(c\kappa^L X_{\tau_1}\cdot l\right)\right]&\\
=\sum_{k \geq 1}\overline{E}_0\big[\exp\big( c\kappa^L X_{S_k}\cdot l\big), S_k<\infty,& D'\circ \theta_{ S_k}=\infty\big]\\
=\sum_{x\in \mathbb Z^d, n\in\mathbb{N}, k\in \mathbb{N}}\mathbb{E}
\big[ E_{Q\times P_{\varepsilon,\omega}^0}\big[e^{c \kappa^L
x\cdot l},X_{S_k}=x,
S_k=&n\big]P_{\theta_n\varepsilon,\theta_x\omega}^0[D'=\infty]\big].
\end{align*}
Notice that for $k\geq1$, the Markov property implies that
\begin{align}
\label{expexp}
&\overline{E}_0\left[\exp \left(c\kappa^L X_{S_k}\cdot l\right), S_k<\infty, D'\circ \theta_{S_k}=\infty\right]\\
\nonumber
&=\sum_{x\in \mathbb Z^d, n\in\mathbb{N}}\mathbb{E} \big[ E_{Q}[ E_{P_{\varepsilon,\omega}^0}[\exp\left(c \kappa^L x\cdot l\right),X_{S_k}=x, S_k=n]P_{\theta_n\varepsilon,\theta_x\omega}^0[D'=\infty]]\big].
\end{align}
Observe now that the random variables
\begin{equation*}
E_{ P_{\varepsilon,\omega}^0}[\exp(c \kappa^L x\cdot
l),X_{S_k}=x, S_k=n]
\end{equation*}
and $P_{\theta_n\varepsilon,\theta_x\omega}^0[D'=\infty]$ are:
$
\sigma\left(\varepsilon_i, i<n\right)\otimes\sigma\left(\omega(y,\cdot), (y-x)\cdot l\leq L | l|_2/| l|_1\right)
$
and
$
\sigma\left(\varepsilon_i, i\geq n\right)\otimes\sigma\left(\omega(y,\cdot), y\in C(x, l, \zeta)\right)
$
measurable, respectively.

\vspace{0.2ex}
\noindent
Therefore for $x\in \mathbb Z^d$, using
the previously introduced notation $\mathfrak{F}_{x,L}$ (cf.
(\ref{sigmafrak})), the mixing condition
\textbf{(SM)}$_{C,g}$ and the
construction of the probability measure $\overline{P}_0$ we find
an $L_0>0$ such that for all $L\geq L_0$, with $L\in
|l|_1\mathbb N$, the rightmost term of (\ref{expexp})
is less than
\begin{align}
\label{splitinequality}
\mathbb{E}\big[ E_{Q\otimes P_{\theta_n\varepsilon, \theta_x\omega}^0}[\mathds{1}_{D'=\infty}]|\mathfrak{F}_{x,L}] \big]&
\leq\overline{E}_0\left[\exp\left( c\kappa^L X_{S_k}\cdot l\right), S_k<\infty\right] \\
\nonumber
\times\exp\bigg(C\sum_{x\in\partial^r(H^c), y \in \partial^r(\Lambda^c)}&e^{-g| x-y|_2}\bigg) P_0[D'=\infty],
\end{align}
where $H$ and $\Lambda$ denote the sets $\{z\in \mathbb{Z}^d:z
\cdot l\leq -L | l|_2/|l|_1 \}$ and $C(0,l,\zeta)$ respectively.
Since $\zeta>0$, the proof of Proposition \ref{propare}
provides the existence of a constant $\widehat{c}>0$ so that
$$
\exp\bigg(C\sum_{x\in\partial
^r(H^c), y \in \partial^r(\Lambda^c)}e^{-g\mid
x-y\mid_2}\bigg)\leq \exp\left(e^{-\widehat{c} L}\right),
$$
with a similar upper bound under \textbf{(SMG)}$_{C,g}$. Going
back to (\ref{expexp}), we have
\begin{align*}
&\overline{E}_0\left[\exp \left(c\kappa^L X_{S_k}\cdot l\right), S_k<\infty, D'\circ \theta_{S_k}=\infty\right]     \\
   & \leq 2\overline{E}_0\left[\exp \left(c\kappa^L X_{S_k}\cdot l\right), S_k<\infty\right]P_0[D'=\infty]].
\end{align*}
We now proceed with the same type of argument of \cite{GR17},
Subsection 6.2; so as to obtain a recursion for $k\geq0$ of the
expression
\begin{equation}
\label{Sk}
\overline{E}_0\left[\exp\left( c\kappa^L X_{S_{k+1}}\cdot l\right), S_{k+1}<\infty\right].
\end{equation}
To this end, it will be convenient to introduce the random variable
$$
M_k:=\sup_{0\leq n\leq R_k} X_n \cdot l,
$$
for $k\geq0$ (with the
convention $M_0=0$). We also introduce the sets parametrized by $k, n \in \mathbb N$:
\begin{equation*}
A_{n,k}=\left\{\varepsilon\in W^{\mathbb N}
:\left(\varepsilon_{t^{(n)}_{k}},\varepsilon_{t^{(n)}_{k}+1},\ldots,
\varepsilon_{t^{(n)}_{k}+L-1}\right)=\overline{\varepsilon}^{(L)}\right\}
\end{equation*}
and:
\begin{equation*}
B_{n,k}=\left\{\varepsilon\in W^{\mathbb N} :
\left(\varepsilon_{t^{(j)}_{k}},\varepsilon_{t^{(j)}_{k}+1},\ldots,\varepsilon_{t^{(j)}_{k
}+L-1}\right)\ne \overline{\varepsilon}^{(L)} \forall
j\in[0,n-1]\right\}.
\end{equation*}
As was mentioned in \cite{GR17}, pp. 25-26; denoting by
$\overline{T}^l_a$ where $a\in \mathbb{R}$ the first time that the
walk goes on strictly over level $a$ in direction $l$, i.e.
$$
\overline{T}^l_a=\inf\{n\geq 0:X_n \cdot l>a\},
$$
and by $(t_k^{(n)})_{n \geq 0}$ the time sequence of successive
maxima in direction $l$, defined recursively via:
\begin{equation*}
t_k^{(0)}=\overline{T}^l_{M_k},\,\,\mbox{ and for }\, n\geq
1:\,\,t_k^{(n)}=\overline{T}^l_{X_{t_k^{(n-1)}}\cdot l},
\end{equation*}
one has the inclusion:
$$
\{S_{k+1}<\infty\}\subseteq\bigcup_{n\geq0}\{t_k^{(n)}<\infty, B_{n,k}, A_{n,k} \}.
$$
Furthermore, $\overline{P}_0$-a.s. on the event $B_{n,k}\cap
A_{n,k}$ the identity $$ S_{k+1}=t_k^{(n)}+L $$ holds. As a result,
we have for $k\geq 0$ the inequality:
\begin{align}
\nonumber
\overline{E}_0\big[\exp \left(c\kappa^L X_{S_{k+1}}\cdot l\right), &S_{k+1}<\infty\big]\\
\nonumber
\leq\sum_{0\leq n\leq L^2-1}\overline{E}_0\big[\exp\big( c\kappa^L X_{S_{k+1}}\cdot& l\big),t_k^n<\infty, B_{n,k},  A_{n,k}\big] \\
\nonumber
+\sum_{n\geq L^2}\overline{E}_0\big[\exp\big( c\kappa^L X_{S_{k+1}}\cdot l&\big),t_k^n<\infty, B_{n,k}, A_{n,k}\big] \\
\label{decomexp1}
\leq 2\sum_{n\geq L^2}\overline{E}_0\big[\exp\big( c\kappa^L
X_{S_{k+1}}\cdot& l\big),t_k^n<\infty, B_{n,k}, A_{n,k}\big],
\end{align}
where the last inequality in (\ref{decomexp1}) can be verified by inspecting  the
orders of $L$ in both sums. Moreover, one can find a positive
constant $\overline{c}$ such that $\overline{P}_0$-a.s. on the
event $\{t_k^{(n)}<\infty, B_{n,k}, A_{n,k}\}$
\begin{equation}
\label{inxsk}
X_{S_{k+1}}\cdot l\leq M_k + n | l|_{\infty}+\overline{c} L,
\end{equation}
holds. Using the product structure of the measure $Q$ and
inequality (\ref{inxsk}), it follows that for $n\geq L^2$
\begin{align*}
&\overline{E}_0\left[\exp\left( c\kappa^L X_{S_{k+1}}\cdot l\right),t_k^{(n)}<\infty, B_{n,k}, A_{n,k}\right]\\
&\leq\kappa^L \overline{E}_0\left[ \exp\left( c\kappa^L(M_k + n | l|_{\infty}+\overline{c} L)\right), t_k^{(n)}<\infty, B_{n,k}\right].
\end{align*}
We now apply the Markov property at times $t_k^{(0)}$ and
$t_k^{(n)}$ (recall that $n \geq L^2$), together with Lemma 6.6 of
\cite{GR17} to see that for some positive constant $\widetilde{c}$,
the inequality:
\begin{align*}
&\kappa^L \overline{E}_0\left[ \exp\left( c\kappa^L(M_k + n \mid l\mid_{\infty}+\overline{c} L)\right), t_k^n<\infty, B_{n,k}\right]\\
&\leq2\kappa^L \left(\exp\left( c | l |_{\infty}\kappa^L
L^2\right)(1-\widetilde{c}L^2\kappa^L)\right)^{[\frac{n}{L^2}]}\overline{E}_0\left[
\exp\left( c\kappa^LM_k\right), t_k^{(0)}<\infty\right]
\end{align*}
holds. Performing summation on $n$ one has that there exists
$\mathfrak{c}>0$ so that
\begin{align}
\label{estimategeom}
&2\sum_{n\geq L^2}\overline{E}_0\left[\exp\left( c\kappa^L
X_{S_{k+1}}\cdot l\right),t_k^{(n)}<\infty, B_{n,k},
A_{n,k}\right]\\
\nonumber
&\leq\mathfrak{c}\kappa^L L^2\frac{1}{\exp\left( -c | l |_{\infty}\kappa^L L^2\right)-(1-\widetilde{c}L^2\kappa^L)}\overline{E}_0\left[\exp\left( c\kappa^L M_k\right), t_k^{(0)}<\infty\right].
\end{align}
It follows that for some small enough constant $c>0$, there
exists $\mathrm{c}>0$ such that
\begin{align*}
&\sum_{n\geq L^2}\overline{E}_0\left[\exp\left( c\kappa^L X_{S_{k+1}}\cdot l\right),t_k^{(n)}<\infty, B_{n,k}, A_{n,k}\right]\\
&\leq\mathrm{c}\overline{E}_0\left[ \exp\left( c\kappa^LM_k\right), t_k^0<\infty\right]\leq
\mathrm{c}\overline{E}_0\left[ \exp( c\kappa^LM_k),R_k<\infty\right]\\
&=\mathrm{c} \overline{E} \left[\exp\left( c\kappa^LX_{S_k}\cdot l\right),S_k<\infty , \exp\left(\kappa^L(M_k-l\cdot X_{S_k})\right), D'\circ\theta_{S_k}<\infty\right].
\end{align*}
Using the Markov property and the product structure of the probability measure $Q$, we have
\begin{align}
\nonumber
\overline{E}_0 \left[\exp\left( c\kappa^LX_{S_k}\right),S_k<\infty , \exp\left(\kappa^L(M_k-l\cdot X_{S_k})\right), D'\circ\theta_{S_k}<\infty\right]&\\
\nonumber
=\sum_{x\in\mathbb Z^d, n\in \mathbb N}\mathbb E \big[ E_{Q\otimes P_{\varepsilon, \omega}^0}[\exp\left(c\kappa^L x\cdot l\right), S_k=n,X_n=x]&\\
\label{Mdinequality}
\times E_{Q\otimes P_{\theta_n \varepsilon, \theta_x\omega}}[\exp\left(c\kappa^L\overline{M}\right), D'<\infty]\big]&,
\end{align}
provided we define:
\begin{equation}
\label{Md}
\overline{M}=\sup_{0\leq n \leq D'}\{(X_n-X_0)\cdot l\}.
\end{equation}
At this point we can apply the same sort of procedure as the one
developed to get the rightmost expression in
(\ref{splitinequality}). More precisely, the last expression in
(\ref{Mdinequality}) can be bounded from above by means of the
following sequence of steps (recall definition (\ref{sigmafrak}),
together with sets $H$ and $\Lambda$, introduced after
(\ref{splitinequality})):
\begin{align}
\nonumber
&\sum_{x\in\mathbb Z^d, n\in \mathbb N}\mathbb E \left[ E_{Q\otimes P_{\varepsilon, \omega}^0}[\exp\left(c\kappa^L x\cdot l\right), S_k=n,X_n=x]\right.\\
\nonumber
&\left. \times E_{Q\otimes P_{\theta_n \varepsilon, \theta_x\omega}}[\exp\left(c\kappa^L\overline{M}\right), D'<\infty]\right]\\
\nonumber
&=\sum_{x\in\mathbb Z^d, n\in \mathbb N}\mathbb E \left[ E_{Q\otimes P_{\varepsilon, \omega}^0}[\exp\left(c\kappa^L x\cdot l\right), S_k=n,X_n=x]\right.\\
\nonumber
&\left.\times\mathbb E\left[E_{Q\otimes P_{\theta_n \varepsilon, \theta_x\omega}}[\exp\left(c\kappa^L\overline{M}\right), D'<\infty]|\mathfrak{F}_{x,L}\right]\right]\\
\nonumber
&\leq\sum_{x\in\mathbb Z^d, n \in \mathbb N}\mathbb E \Bigg[ E_{Q\otimes P_{\varepsilon, \omega}^0}[\exp\left(c\kappa^L x\cdot l\right), S_k=n,X_n=x]\\
\nonumber
&\times\exp\Big(C\sum_{x\in\partial^r(H^c), y \in \partial^r(\Lambda^c)}e^{-g| x-y|_1}\Big)\times \overline{E}_0\left[\exp\left(c\kappa^L\overline{M}\right), D'<\infty\right]\Bigg]\\
\nonumber
&\leq 2\overline{E}_0\left[\exp\left(c\kappa^LX_{S_k}\cdot l\right),
S_k<\infty\right]\times
\overline{E}_0\left[\exp\left(c\kappa^L\overline{M}\right),
D'<\infty\right].
\end{align}
Thus an induction argument makes us conclude that for a suitable
constant $\mathrm{c}>0$,
\begin{align*}
&\overline{E}_0\left[\exp\left(c\kappa^LX_{S_{k+1}}\cdot l\right), S_{k+1}<\infty\right] \\
&\leq\left(\overline{E}_0\left[\mathrm{c}\exp\left(c\kappa^L\overline{M}\right), D'<\infty\right]\right)^k\times \overline{E}_0\left[\exp\left(c\kappa^LX_{S_{1}}\cdot l\right), S_{1}<\infty\right].
\end{align*}
On the other hand, for $k=0$, the inequality (\ref{inxsk}) is still
being true. As a consequence, one can obtain the same upper bound as in
the rightmost expression of (\ref{estimategeom}) when $k=0$ (which
implies in turn that $M_0=0$). Hence as a result,
\begin{equation}
\label{recursionSk}
\overline{E}_0\left[\exp\left(c\kappa^LX_{S_{k}}\cdot l\right), S_{k}<\infty\right]\leq \left(E\left[\mathbf{c}\exp\left(c\kappa^L\overline{M}\right), D'<\infty\right]\right)^{k},
\end{equation}
holds.

\vspace{0.5ex}
\noindent
The following auxiliary result will finish the proof.
\begin{lemma}[under \textbf{(T)}$_\ell$]
\label{expintM}
There exist constants $c_4, c_5>0$, such that
\begin{equation}
\label{expbM}
E_0\left[\exp (c_4\overline{M}), D'<\infty\right]<c_5.
\end{equation}
\end{lemma}
\begin{proof}
We observe that replacing $c$ by $c/| l |_2$ below, it will be
sufficient to prove that for some $c>0$, there exists finite $c'>0$
such that
\begin{equation*}
E_0\left[\exp \left(cM'\right), D'<\infty\right]<c',
\end{equation*}
where as a matter of definition, we have denoted by:
\begin{equation*}
M'=\sup_{0\leq n \leq D'}\{(X_n-X_0)\cdot \ell\} \,\,\mbox{(cf. (\ref{rul})).}
\end{equation*}
Notice that
\begin{align}
\nonumber
&E_0[\exp\left(cM'\right), D'< \infty ]\leq e^cP_0[D'< \infty ]+\\
\nonumber
&\sum_{m\geq0} \exp\left(c2^{m+1}\right)P_0[2^m\leq M'<2^{m+1},D'< \infty ].
\end{align}
As a consequence of the previous decomposition inequality, it suffices to
obtain an appropriate upper bound for large $m$ of the probability:
$$
P_0[2^m\leq M'<2^{m+1},D'< \infty ].
$$
To this end, it will be convenient to introduce the
following stopping time for the canonical filtration of the walk:
\begin{equation}
\label{D0}
D'(0)=\inf\{n\geq0: X_n\notin C(0,l,\zeta)\}
\end{equation}
Plainly, using the notation of (\ref{rightst})-(\ref{enandex}) one has the
inequality:
\begin{align}
\label{decompositionM}
&P_0[2^m\leq M'<2^{m+1},D'< \infty ]\\
\nonumber
&\leq P_0[T^\ell _{2^m} \leq D'<\infty, T^\ell_{2^{m+1}}\circ\theta_{T^\ell_{2^m}}
> D'(0)\circ\theta_{T^\ell_{2^m}}]\\
\nonumber
&\leq P_0[ X_{T^\ell_{2^m }}\not \in \partial^+B_{2^m,\mathfrak{r}2^m,\ell}(0),  T^\ell_{2^m } \leq D'<\infty ]\\
\nonumber
&+P_0[X_{T^\ell_{2^m}}\in
\partial^+B_{2^m,\mathfrak{r}2^m,\ell}(0), T^\ell_{2^{m+1}}\circ\theta_{T^\ell_{2^m}}> D'(0)\circ\theta_{T^\ell_{2^m}}].
\end{align}
Notice that on the event of the first probability on the rightmost
expression in (\ref{decompositionM}), $P_0$-a.s. one has
\begin{equation}
X_{T_{B_{2^m,\mathfrak{r}2^m,\ell}(0)}} \notin \partial^+B_{2^m,\mathfrak{r}2^m,\ell}(0).
\end{equation}
Therefore, condition \textbf{(T)}$_\ell$ implies that for large $m$,
\begin{align}
\nonumber
P_0[ X_{T^\ell_{2^m }}\not \in
\partial^+B_{2^m,(\frac{2}{\varepsilon})2^m,\ell}(0), T^\ell_{2^m } \leq
D'<\infty ]\\
\label{expb1}
\leq \exp\left(-\mathfrak{c}2^m\right)
\end{align}
for some suitable positive constant $\mathfrak{c}$. As for the
second term on the rightmost expression of (\ref{decompositionM}),
for $m\in \mathbb N$ we introduce the boundary box $F_{m}$ via:
$$
F_{m}=\partial^+B_{2^m,\mathfrak r 2^m,\ell}(0).
$$
Applying the strong Markov property we find that
\begin{align}
\nonumber
&P_0[X_{T^\ell _{2^m}}\in
\partial^+B_{2^m,\mathfrak r 2^m,\ell}(0), T^\ell
_{2^{m+1}}\circ\theta_{T^\ell_{2^m}}>
D'(0)\circ\theta_{T^\ell_{2^m}}]\\
\label{expb2}
&\leq \sum_{y\in F_{m}}P_y[T^\ell _{2^{m+1} }>D'(0)].
\end{align}
In order to estimate the rightmost probability entering in
(\ref{expb2}), we will bound from below the probability of its
complementary event as follows. Introducing for $x\in\mathbb{Z}^d$,
the set:
\begin{equation}
B_x=B_{2^{m-1},\mathfrak{r}2^{m-1},\ell}(x),
\end{equation}
we note that under the assumption (\ref{zeta}) we have
$$
\mathfrak{r}\left(2^m+2^{m-1}\right)\leq \tan \left( \frac{\pi}{2}-\arccos(\zeta)\right)\,2^{m-1},
$$
which implies that the boxes $B_y$ and $B_z$,
where $y\in F_{m}$ and $z\in \partial^+ B_y$, are both inside of the cone
$C(0,l,\zeta)$ (see Figure \ref{fig} below).
\begin{figure}[h]
  \begin{center}
  \includegraphics[width=7cm]{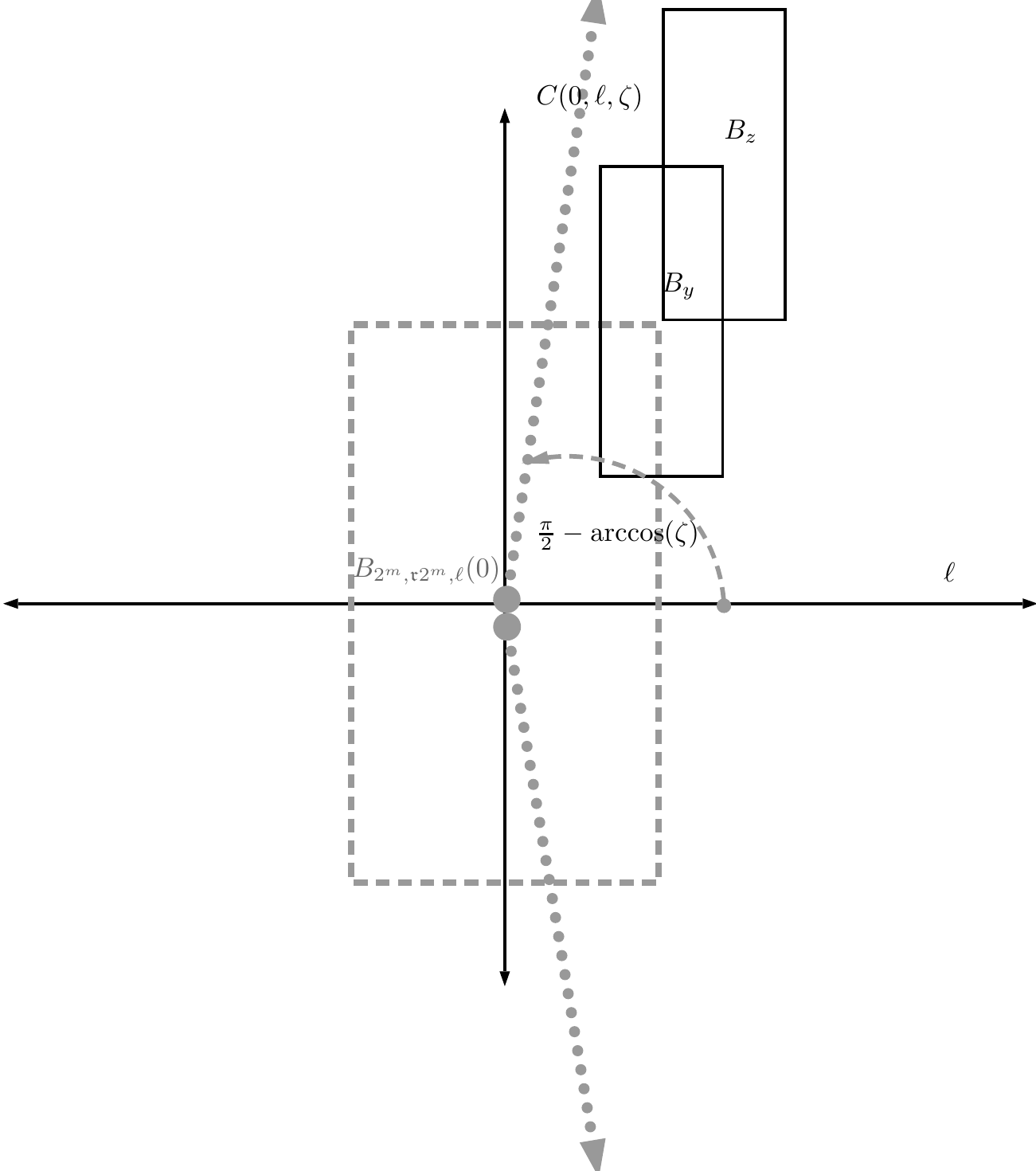}\\
  \caption{Boxes $B_y$ and $B_z$ are inside of $C(0,l,\zeta)$.}\label{fig}
  \end{center}
\end{figure}

\medskip
\noindent
Observe that for $y \in F_m$, one has the following lower bound:
\begin{align}
\nonumber
 & P_y[T^\ell_{2^{m+1}} < D'(0)]\\
\label{Pyc2}
&\geq \sum_{z \in  \partial^+ B_y}\mathbb{E}[P_{y,\omega}[X_{T_{B_y}}\in
\partial^+ B_y, X_{T_{B_y}}=z,
(X_{T_{B_z}}\in\partial^+ B_z)\circ\theta_{T_{B_y}}]].
\end{align}
To estimate the right-hand side of the above inequality, it will be
convenient to introduce for $m\in \mathbb N$, the \textit{second
boundary} set $\bar F_{m}$ as
$$
\bar F_{m}:=\partial[\cup_{y\in F_{m}}B_y]\cap  R([ 2^{m-1}+2^m,\infty)\times \mathbb{R}^{d-1}),
$$
and in turn for that given set $\bar F_{m}$ we introduce the
\textit{good environment} event $G_{\bar F_{m}}$ by
\begin{align*}
G_{\bar F_{m}}:=&\{\omega\in \Omega : \ P_{z,\omega}[X_{T_{B_z}}\in
\partial^+B_z]>\\ &1-\exp(-\mathrm{c} 2^{(m-1)}), \ \mbox{for all
}z \in \bar F_{m} \},
\end{align*}
where the constant $\mathrm{c}>0$ will be chosen below.
Using the strong Markov property, we can now bound from below the
right-hand side of inequality (\ref{Pyc2}) by
\begin{equation}
\label{expression}
\left(1-\exp(-\mathrm{c} 2^{(m-1)})\right)\left(P_y[X_{T_{B_y}}\in \partial^+
B_y]- P_y[(G_{\bar F_{m}})^c]\right),
\end{equation}
where for an event $E$, we denote by $(E)^c$ its complementary
event.

\vspace{0.5ex}
\noindent
Furthermore, using stationarity under the probability measure
$\mathbb P$ and condition \textbf{(T)}$_\ell$, for $x\in\mathbb R^d$ and
large $m$ one has
\begin{align}
\nonumber
P_{x}[X_{T_{B_x}}\not\in \partial^+B_x]=P_{0}[X_{T_{B_0}}\not\in
\partial^+B_0] \\
\label{k1}
\leq \exp\left(-\mathfrak{w}2^{m-1}\right),
\end{align}
for a suitable $\mathfrak{w}>0$.

\noindent
We thus see that (\ref{expression}) is greater than
\begin{equation}
\label{in1}
\left(1-\exp\left(-\mathrm{c} 2^{(m-1)}\right)\right)\left(1-\exp\left(-\mathfrak{w}2^{m-1}\right)-P_y[(G_{\bar F_{m}})^c]\right)
\end{equation}

\noindent
Taking $\mathrm{c}=\mathfrak{w}/2$, in virtue of (\ref{k1}) and
Chevyshev's inequality we find that
\begin{align}
P_y[(G_{\bar F_{m}})^c]&\\
\nonumber
&\leq|\bar F_{m}|\exp\left(\mathrm{c}2^{(m-1)}\right)\sup_{x\in
\bar F_{m}}P_{x}[ X_{T_{B_x}}\not\in
\partial^+B_x]\leq\exp\left(-\mathfrak{t}2^{m-2}\right),
\label{estimatek1}
\end{align}
for a suitable $\mathfrak{t}>0$, where we have used for $m\in \mathbb N$ the coarse estimate:
$$
\max\left(|\bar F_{m}|, |F_m| \right)\leq \left(6\mathfrak{r}2^{m}\right)^{d-1}.
$$
Consequently, for large $m$ we can find a further positive
constant $\widetilde{c}$ such that:
\begin{equation}
\label{Destimate}
P_y[T^\ell_{2^{m+1}} \leq D'(0)]\geq 1-\exp(-\widetilde{c}2^m)
\end{equation}
for all $y\in \bar{F}_m$.

\vspace{0.2ex}
In view of (\ref{expb1}), (\ref{expb2}) and (\ref{Destimate}), the
claim (\ref{expbM}) follows.
\end{proof}
As it was mentioned the assertion in (\ref{expmompos}) follows from
(\ref{expbM}) and $P_0[D'=\infty]>0$, with the help of estimate
(\ref{recursionSk}).
\end{proof}

\medskip
We are now ready to spell out some consequences of the previous
proposition. We first define the random variable $Y$ as
\begin{equation}
\label{supt1}
Y=\sup_{\substack{0\leq n \leq \tau_1}}|X_n|_2.
\end{equation}
We can prove the following reinforcement to Theorem \ref{expmpr}:
\begin{corollary}[under \textbf{(T)}$_\ell$]\label{corexp}
Assume either: \textbf{(SM)}$_{C,g}$ or \textbf{(SMG)}$_{C,g}$. Then there exist positive constants $c_{6}$, $c_{7}$ and $L_0$ such that
\begin{equation}
\label{esupt1}
\overline{E}_0[e^{c_{6}\kappa^L Y}]\leq c_{7}.
\end{equation}
provided that $L\geq L_0$, $L\in |l|_1\mathbb N$.
\end{corollary}
\begin{proof}
Using item (ii) of Lemma \ref{T}, notice that for large $u$,
\begin{align}
  \nonumber
  &\bar{P}_0\left[Y\geq u\right]=\bar{P}_0\left[\sup_{0\leq n\leq \tau_1}|X_n|_2\geq u\right] \\
  \nonumber
  &\leq \bar{P}_0\left[T_{\Delta_{\frac{u}{2\hat{r}}}}<\tau_1\right] \\
  \nonumber
  &\leq \bar{P}_0\left[X_{\tau_1}\cdot l\geq \frac{u}{2\hat{r}} \right]+\bar{P}_0\left[X_{\tau_1}\cdot l<\frac{u}{2\hat{r}}, \, T_{\Delta_{\frac{u}{2\hat{r}}}}<\tau_1\right]\\
  \label{estimexpT}
  &\leq \exp\left(-\kappa^L\,c_2\frac{u}{2r}\right)\bar{E}_0\left[\exp\left(c_2\kappa^L\,X_{\tau_1}\cdot l\right)\right]+P_0\left[X_{T_{\Delta_{\frac{u}{2\hat{r}}}}}\notin \partial^+\Delta_{\frac{w}{2\hat{r}}}\right],
\end{align}
where in the last step we have used that by definition $X_m\cdot
l<X_{\tau_1}\cdot l$, when $0\leq m <\tau_1 $. Keeping in mind the
\textit{layer cake decomposition} (cf. \cite{Ru87}, Chapter 8, Theorem 8.16), the claim of the corollary
follows after applying condition \textbf{(T)}$_\ell$ and Proposition \ref{expmpr}.
\end{proof}

In order to state the next proposition it will be useful to fix
some further notation. For $L\in |l|_1\mathbb N$, we introduce
\textit{the approximate asymptotic direction} denoted as $\hat{v}_L\in \mathbb S^{d-1}$ and given by
\begin{equation}
\label{asym}
\hat{v}_L:=\frac{\overline{E}_0[X_{\tau_1}| D'=\infty]}{|\overline{E}_0[X_{\tau_1}| D'=\infty]|}.
\end{equation}
which a priori depends on $L$, however when there is not risk of
confusion, we shall drop it.

\vspace{1ex}
As explained in \cite{GR17} Proposition 7.2. page 34, one has:
\begin{proposition}
\label{remark}
There exist positive constants $k_1$ and $k_2$ (not depending on
$L$) such that for any $L\in |l|_1\mathbb N$,
\begin{align*}
&\overline E_0[(\kappa^L\,X_{\tau_1}\cdot l)|D'=\infty]\geq k_1 \hspace{1ex}\mbox{ and}
&|\overline E_0[(\kappa^L\,X_{\tau_1})|D'=\infty]|_2\geq k_2.
\end{align*}
\end{proposition}
Thus, the upper bounds obtained in this sections are sharps. Proposition \ref{remark} will be useful to prove Theorem \ref{th1} in Section \ref{seclln}.

\vspace{1ex}
We continue with the definition for $t\in \mathbb{R}$ of the random
variable
\begin{equation}
\label{lastvisit}
M_t:=\sup\{n \geq 0: X_n \cdot l\leq t\},
\end{equation}
this is \text{the last visit} to the half space $H=\{z: z
\cdot l \leq t\}$. We also define the projector operator
$\Pi=\Pi_{\hat{v}}:\mathbb R^d\,\rightharpoonup \, \mathbb R^d$
onto the orthogonal space to $\hat{v}_L$, so that for $z\in
\mathbb{R}^d$
\begin{equation*}
\Pi(z)=z-(z\cdot \hat{v})\hat{v}.
\end{equation*}
The next proposition will be fundamental to apply renormalization
arguments in order to obtain annealed estimates of atypical quenched
escapes for the walk.

\medskip
\begin{proposition}[under \textbf{(T)}$_\ell$, see (\ref{rul})]
\label{asymtrafluc}
Let $C,g>0$ and assume either: \textbf{(SM)}$_{C,g}$ or \textbf{(SMG)}$_{C,g}$ and (\ref{Rgkappa}).
Let $\gamma\in(5/9,1)$ and $\rho>0$. Then there exists $c_8=c_8(d, \rho, \kappa, l)>0$, so that
for large $u$ one has that
\begin{equation}
\label{projection}
P_0\left[\sup_{0\leq n\leq M_{u}}|\Pi_{\hat{v}}(X_n)|\geq \rho
u^\gamma\right]\leq \exp\left(-c_8
u^{\frac{9}{4}\gamma-\frac{5}{4}}\right),
\end{equation}
with the notation as in (\ref{lastvisit}) and
$\hat{v}=\hat{v}_L$ is the vector defined by
(\ref{asym}), where for a fixed number $t\in(1/2,1)$ with
\begin{equation}\label{choiset}
gt>18\log\left(\frac{1}{\kappa}\right),
\end{equation}
$L$ is the least integer in $|l|_1\mathbb N$, such that:
\begin{equation*}
\exp\left(-g\, tL\right)\leq u^{2\gamma-2+\frac{\gamma-1}{4}}.
\end{equation*}
\end{proposition}
\begin{proof}
 Fix $\gamma \in (5/9, 1)$, $t$ as in (\ref{choiset}) and consider a large enough $u$ so that the least integer $L\in |l|_1\mathbb N$ satisfying
\begin{equation}\label{minimazerL}
\exp\left(-g\, tL\right)\leq u^{2\gamma-2+\frac{\gamma-1}{4}},
\end{equation}
is also satisfying the following requirements:
\begin{align}
\label{constraintL1}
&\,\,L\geq L_0,\,\,\\
\label{constraintL2}
&\,\,L\geq
\frac{6(2c_7+1)|l|_1e^{\frac{g|l|_1}{36}}}{|l|_2\rho\, c_6}\,\,
\mbox{ and }\\
\label{constraintL3}
&\,\,u^{2(\gamma-1)+(\frac{\gamma
-1}{4})}\leq\frac{1}{2}.
\end{align}
Above, constant $L_0$ is as in the statement of
Corollary \ref{corexp}. For the rest of the proof, we will drop the
prescribed $L$ defined by (\ref{minimazerL}) and satisfying (\ref{constraintL1})-(\ref{constraintL3}))
from the notation, to set for instance: $\hat{v}=\hat{v}_L, \,
\tau_1=\tau_1^{(L)},$ and so on. Furthermore, notice that it is
sufficient to prove an analogue inequality to (\ref{projection}),
replacing $\Pi_{\hat{v}}(X_n)$ by $X_n\cdot w$, where $w\in \mathbb
S^{d-1} $ with $w \cdot \hat{v}=0$. Therefore, we will prove the
proposition under this convention and we introduce for $n\in
\mathbb N$ the random variable $K_n$, via
\begin{equation*}
K_n=\sup\{k\geq 0: \tau_k \leq n\}\,\, \mbox{ (set $\tau_0=0$.)}
\end{equation*}
Since $\overline{P}_0$-a.s. one has for $m\leq \tau_1 \leq m'$:
\begin{equation*}
X_m \cdot l\leq X_{\tau_1}\cdot l\leq X_{m'}\cdot l \,\, \mbox{ and
}\,\, X_{\tau_1}\cdot l\geq L \frac{|l|_2}{|l|_1},
\end{equation*}
it follows that $\overline{P}_0$-a.s.
\begin{equation}
\label{boundfK}
0\leq n \leq M_{u} \, \Rightarrow \,K_n\leq \frac{|l|_1}{|l|_2 L} u.
\end{equation}
Hence, for $n\in[0,M_{u}]$ and $K_n$ as above, we have (recall the
notation in display (\ref{supt1}))
\begin{equation*}
X_n\cdot w =(X_n-X_{\tau_{K_n}})\cdot w +X_{\tau_{K_n}}\cdot w\leq
Y\circ \theta_{K_n}+ X_{\tau_{K_n}}\cdot w,
\end{equation*}
and consequently for $\rho>0$ we get the inequality:
\begin{align*}
 P_0\left[\sup_{0\leq n \leq M_{u}} X_n\cdot w\geq \rho u^\gamma \right]\leq &
  \sum_{0\leq k \leq \frac{|l|_1}{|l|_2 L} u}\overline{P}_0\left[Y\circ \theta_{\tau_k}\geq \frac{\rho}{3}u^\gamma\right]\\
  +\overline{P}_0\left[X_{\tau_1}\cdot w \geq \frac{\rho}{3}u^\gamma\right]+ \sum_{2\leq k \leq \frac{|l|_1}{|l|_2 L}u}& \overline{P}_0\left[ (X_{\tau_k}-X_{\tau_1})\cdot w\geq\frac{\rho}{3}u^\gamma\right].
\end{align*}
Let $\lambda \in [0, c_{6} \kappa^L] $ and observe that an
application of Chernoff bound leads us
to
\begin{gather}
\nonumber
 P_0\left[\sup_{0\leq n \leq M_{u}} X_n \cdot w \geq \rho  u^{\gamma} \right] \leq \exp\left(-\lambda \frac{\rho}{3}u^{\gamma}\right)
\left( \sum_{0\leq k \leq \frac{|l|_1}{|l|_2 L}u}\overline{E}_0[\exp\left(\lambda Y\circ \theta_{\tau_k}\right)]\right.\\
\label{estimportant}
\left. +\overline{E}_0[\exp\left(\lambda X_{\tau_1}\cdot w\right)]+
\sum_{2\leq k \leq \frac{|l|_1}{|l|_2 L} u}\overline{E}_0[\exp\left(\lambda(X_{\tau_k}-X_{\tau_1})\cdot w \right)]\right).
\end{gather}
Let us now perform some computations required to estimate the
expectations entering in the last expression above. We first
observe that for integer $k\geq0$
\begin{align*}
\overline{E}_0\left[\exp(\lambda Y\circ \theta_{\tau_k})\right]&=\sum_{k\geq 1,n\in \mathbb N, x\in \mathbb Z^d}\mathbb E\left[E_{Q\otimes P_{\varepsilon, \omega}^0}\left[\mathds{1}_{S_k=n, X_{S_k}=x}\right]\right.\\
&\left. \times E_{Q\otimes P_{\theta_n\varepsilon,\theta_x \omega}^0}\left[\exp(\lambda Y),D'=\infty\right]\right]\\
&=\sum_{k\geq 1,n\in \mathbb N, x\in \mathbb Z^d}\mathbb E\left[E_{Q\otimes P_{\varepsilon, \omega}^0}\left[\mathds{1}_{S_k=n, X_{S_k}=x}\right]\right.\\
&\left.\times\mathbb{E}\left[ E_{Q\otimes P_{\theta_n\varepsilon,\theta_x \omega}^0}\left[\exp(\lambda Y),D'=\infty\right]| \mathfrak{F}_{x,L}\right] \right].
\end{align*}
Using the proof of the Proposition \ref{propare}, it is easy to see
that for the non-negative random variable $\lambda Y$, the
inequality
\begin{align*}
&\mathbb{E}\left[ E_{Q\otimes P_{\theta_n\varepsilon,\theta_x \omega}^0}\left[\exp(\lambda Y)|D'=\infty\right]| \mathfrak{F}_{x,L}\right]\\
&\leq \exp\left(e^{-g\,tL}\right)\overline{E}_0[\exp(\lambda Y)|D'=\infty]
\end{align*}
holds.

\noindent
Therefore, as a result we get for integer $k\geq0$ the estimate
\begin{align}
\nonumber
&\max\left\{\overline{E}_0\left[\exp(\lambda Y\circ \theta_{\tau_k})\right], \overline{E}_0\left[\exp(\lambda X_{\tau_1}\cdot w)\right]\right\}\\
\label{est1}
&=\overline{E}_0\left[\exp(\lambda Y\circ
\theta_{\tau_k})\right]\leq 2 \overline E_0[\exp(\lambda Y)|D'=\infty].
\end{align}

\noindent
On the other hand, quit a similar procedure but now using the
complete statement of Proposition \ref{propare} along successive
conditioning, allows us to conclude that for $k\in [2,|l|_1 u/|l|_2 L]$
one has:
\begin{align}
\nonumber
\overline{E}_0[\exp(\lambda(X_{\tau_k}-X_{\tau_1})\cdot w )]=&\overline{E}_0[\exp(\lambda  \sum_{j=2}^k (X_{\tau_j}-X_{\tau_{j-1}})\cdot w)]\\
\nonumber
\leq&\left(\exp\left(e^{-g\, tL}\right)\overline{E}_0[\exp(\lambda  X_{\tau_1}\cdot w)| D'=\infty]\right)^{k-1}\\
\label{est2}
\stackrel{(\ref{boundfK})}\leq & \left(\exp\left(e^{-g\, tL}\right)\overline{E}_0[\exp(\lambda  X_{\tau_1}\cdot w)| D'=\infty]\right)^{\frac{|l|_1 u}{|l|_2 L}}.
\end{align}
Define now for $|\lambda| \leq \kappa^L c_{6}$, the function
\begin{equation*}
  H(\lambda):= \overline{E}_0[\exp\{\lambda X_{\tau_1}\cdot w \}|D'=\infty].
\end{equation*}

Taking $\lambda=\varrho u^{\frac{5}{4}(\gamma-1)}$ for a
positive constant $\varrho$ chosen so that
\begin{equation}\label{varrho}
\frac{c_6}{2}\,e^{-g\frac{|l|_1}{36}}<\varrho <c_6\,e^{-g\frac{|l|_1}{36}},
\end{equation}
holds, from the very definition of $L$ in (\ref{minimazerL}), we obtain
\begin{equation*}
 \varrho u^{\frac{\gamma -1}{8}}<\varrho e^{-\frac{g\,t(L-|l|_1)}{18}}\leq c_6 \,e^{-\log(1/\kappa)L}.
\end{equation*}
We observe that, for our choice of $\lambda$, $w\,\bot\,\hat{v}$,
Proposition \ref{expmpr} and Lebesgue's dominated convergence
theorem, one has:
\begin{align*}
H(\lambda)&=\overline{E}_0\left[1+\lambda X_{\tau_1}\cdot
w+\frac{\lambda^2}{2!}( X_{\tau_1}\cdot w)^2+\frac{\lambda^3}{3!}(
X_{\tau_1}\cdot w)^2+\ldots|D'=\infty \right]\\
&\leq \overline E_0[1+\varrho u^{\gamma-1+(\frac{\gamma -1}{4})}X_{\tau_1}\cdot w+ u^{2(\gamma -1)+(\frac{\gamma -1}{4})}\frac{(\varrho u^{\frac{\gamma -1}{8}}X_{\tau_1}\cdot w)^2}{2!}\\
&+ u^{3(\gamma-1)+(\frac{2(\gamma-1)}{4})}\frac{(\varrho u^{\frac{\gamma -1}{8}}X_{\tau_1}\cdot w)^3}{3!} +\ldots|D'=\infty]\\
 &\leq 1+c_7 u^{2(\gamma-1)+\left(\frac{\gamma -1}{4}\right)}\sum_{j=0}^\infty u^{j\left((\gamma-1)+\left(\frac{\gamma -1}{4}\right)\right)}\\
&\stackrel{\ref{constraintL3}}\leq  1 + 2c_7 u^{2(\gamma-1)+(\frac{\gamma -1}{4})}\leq e^{2c_7u^{2(\gamma-1)+\left(\frac{\gamma -1}{4}\right)}}.
\end{align*}
Consequently, once again since (\ref{minimazerL}) and
requirement (\ref{constraintL1}) we have that
\begin{align}
\nonumber
(\ref{est2})=&\left(\exp\left(e^{-g\,
tL}\right)\exp\left(\log(\overline{E}_0[\exp(\lambda
X_{\tau_1}\cdot w)| D'=\infty])\right) \right)^{\frac{|l|_1
u}{|l|_2 L}}\\
\label{extk}
\leq & \exp\left((2c_7+1)u^{2\gamma-2+\left(\frac{\gamma-1}{4}\right)}\times\left(\frac{|l|_1 u}{|l|_2 L}\right)\right).
\end{align}
Inserting estimates (\ref{est1}) and (\ref{extk}) into
(\ref{estimportant}) the assertion of the proposition follows since
assumption (\ref{constraintL2}).
\end{proof}

\begin{remark}\label{remfdre}
Let us sketch the proof for finite dependent random environments. Taking $L$ large enough with respect to the dependence of the environment we get to the rightmost expression in (\ref{est2}) without factor
$\exp\left(e^{-g\,tL}\right)$. Then, it is direct to see that i.i.d. renormalization techniques can be applied in this case without the help of assumption (\ref{Rgkappa}). Here, the crucial point is that there exists a finite $L$ such that $\tau_1^{(L)}$ is in fact a regeneration time.
\end{remark}

\section{Estimates for the Regeneration Time Tails}
\label{sectailest}

The main objective of this section will be to obtain an upper bound
for the probability $\overline{P}_0[\tau_1\geq u]$ when $u$ is
large and independent of $L$. Let $C,g>0$, throughout the complete section we shall assume
condition \textbf{(T)}$_\ell$, where $\ell\in \mathbb S^{d-1}$ satisfies (\ref{rul}), and either:
\textbf{(SM)}$_{C,g}$ or \textbf{(SMG)}$_{C,g}$. We first prove a basic lemma in
the spirit of \cite{Sz00}, Lemma 1.3. It is convenient to fix a
rotation $R$ on $\mathbb R^d$, with
$$
R(e_1)=\frac{l}{|l|_2}=\ell.
$$
Introducing for $M>0$, the hypercube
\begin{equation}
\label{cube}
C_M:=B_{M,\mathfrak{r} M, \ell}(0).
\end{equation}
We have
\begin{lemma}\label{basicexpest}
There exist $c_9>0$ and $L_0>0, \, L_0\in |l|_1\mathbb N$ such that
for any function $M:\mathbb R^+\rightarrow\mathbb R^+$, with
$\lim_{\substack{u\rightarrow\infty}} M(u)=\infty$ one has that for
large $u$,
\begin{equation*}
\overline{P}_0[\tau_1>u]\leq P_0[T_{C_{M(u)}}=T_{M(u)}^l>u]+e^{-c_9\kappa^L M(u)}
\end{equation*}
for each $L\in|l|_1\mathbb N, \, L\geq L_0$.
\end{lemma}
\begin{proof}
Let us start with the inequality
\begin{align}
\nonumber
  \overline{P}_0[\tau_1>u]\leq \overline{P}_0[\tau_1>u, X_{\tau_1}\cdot l&\leq |l|_2 M(u)]+\overline{P}_0[X_{\tau_1}\cdot l> |l|_2 M(u)]  \\
  \nonumber
  \leq \overline{P}_0[\tau_1>u, X_{\tau_1}\cdot l\leq |l|_2 M(u)] &+e^{-c_1\kappa^L |l|_2 M(u)}\overline{E}_0\left[\exp\left(c_1\kappa^LX_{\tau_1}\cdot l\right)\right]\\
  \label{firstdecom}
  \leq \overline{P}_0[\tau_1>u,& X_{\tau_1}\cdot l\leq |l|_2 M(u)]+e^{-\frac{c_1|l|_2\kappa^L M(u)}{2}}.
\end{align}
It is then sufficient to estimate the probability
$$
\overline{P}_0[\tau_1>u, X_{\tau_1}\cdot l\leq |l|_2 M(u)].
$$
From the definition of time $\tau_1$, one has that
$\tau_1=T_{X_{\tau_1}\cdot l}^l$. Hence, we find that
\begin{equation*}
\overline{P}_0[\tau_1>u, X_{\tau_1}\cdot l\leq M(u)]\leq P_0[T_{|l|_2 M(u)}^l > u]\stackrel{(\ref{rightst})}=P_0[T_{M(u)}^{\ell}>u].
\end{equation*}
We first proceed to consider the following decomposition
inequality
\begin{equation*}
P_0[T_{M(u)}^{\ell} > u]\leq P_0[T_{C_{M(u)}}=T_{M(u)}^{\ell} >
u]+P_0[T_{C_{M(u)}}<T_{M(u)}^{\ell}],
\end{equation*}
for large $u$. Since \textbf{(T)}$_\ell$ holds, (see (\ref{generalboxes}) and Lemma \ref{T})
\begin{align}
\nonumber
P_0[T_{C_{M(u)}}<T_{M(u)}^{l'}]&\leq P_0\left[X_{T_{C_{M(u)}}}\notin \partial^+C_{M(u)}\right]\\
\label{ineqkali}
&\leq \exp\left(-\widetilde{c}M(u)\right),
\end{align}
for a suitable constant $\widetilde{c}>0$.

\noindent
Thus, coming back to (\ref{firstdecom}) the required assertion follows from
(\ref{ineqkali}).
\end{proof}
In the next subsection we will present an atypical
quenched estimate for mixing environments in the spirit of \cite{Sz00}, Proposition 3.1.

\subsection{Renormalization}
The main objective here is to establish a version
of an atypical quenched estimate for mixing random environments in the spirit of Proposition 3.1 in \cite{Sz00} for i.i.d. environments. To this purpose, we first introduce the
set
\begin{equation*}
  U_M=\left\{y\in \mathbb Z^d: \left|y\cdot \ell\right|<M\right\}
\end{equation*}
for $M>0$. The crucial ingredient to bound from above the tail of
$\tau_1$ is given below.
\begin{proposition}
\label{atyquenchedestimate}
For $\beta\in [0,1)$ and $c>0$
\begin{equation}
\label{aqe}
\limsup_{\substack{M\rightarrow\infty}} M^{-\chi}\log\mathbb P\left[P_{0,\omega}\left[X_{T_{U_M}}\cdot \frac{l}{|l|_2}\geq M\right]\leq e^{-cM^\beta}\right]<0,
\end{equation}
\begin{equation}\label{conaqe}
\mbox{where either $\chi=1$ or $\chi<d(\frac{13}{4}\beta-\frac{9}{4})$.}
\end{equation}
\end{proposition}
\begin{proof}
By a quit similar argument of \cite{Sz00}, page 121, the case
$\chi=1$ easily follows from condition \textbf{(T)}$_{\ell}$. We thus
only need to consider the case when $\beta\in[0,1)$ is large enough
such that
\begin{equation}\label{betacon}
d\left(\frac{13}{4}\beta-\frac{9}{4}\right)>1.
\end{equation}
The key idea of the proof (cf. \cite{Sz00}) is to construct strategies
for the walk ensuring that this starting from $0\in \mathbb Z^d$,
escapes from $U_M$ by the boundary side $\partial^+U_M:=\partial
U_M \cap \{z\in \mathbb R^d: z\cdot l/|l|_2\geq M\}$. Such a
construction involves the notion of \textit{good} and \textit{bad}
boxes for the environment, and they will provide high probability
on the event that the walk fulfills the required strategies.

\noindent
In order to introduce the definitions of good and bad boxes, we
need some further notation. For $L\geq L_0$ with $L\in |l|_1
\mathbb N$, we pick a rotation $\tilde{R}_L$ on $\mathbb R^d$ so
that
\begin{equation*}
  \tilde{R}_L(e_1)=\hat{v}_L \,\,\mbox{ (we shall only write $\tilde R$, beacuse we will fix $L$ below). }
\end{equation*}

\noindent
We consider $\gamma\in(5/9,1)$ and $t\in (1/2,1)\cap \mathbb Q$ so
that
$$
tg>18\log\left(\frac{1}{\kappa}\right).
$$
Pick then
$M_0>2\sqrt{d}$ large enough, such that if $L$ is the integer
satisfying $$ L=\min\{\widehat{L}\in |l|_1 \mathbb N:
e^{-gt\widehat{L}}\leq M_0^{2\gamma -2 +\left(\frac{\gamma
-1}{4}\right)}\}, $$ one has that $L\geq L_0$ and $L\geq
\frac{48|l|_1}{\hat{v}\cdot l}$ (which is possible by
Proposition \ref{remark}).

\vspace{1ex}
\noindent
Define for $z\in M_0\,\mathbb Z^d$ ($M_0$ as above), the following
blocks:
\begin{align}
\nonumber
\tilde{B}_1(z)&:=\tilde{R}\left(z+\left(0, M_0\right)^d\right)\cap \mathbb Z^d\\
\label{blocks}
\tilde{B}_2(z)&:=\tilde{R}\left(z+\left(-M_0^\gamma, M_0+M_0^\gamma\right)^d\right)\cap \mathbb Z^d,
\end{align}
which are nonempty because $M_0>2\sqrt{d}$. One also defines the
\textit{boundary positive part of} $\tilde{B}_2(z)$ via
\begin{equation}
\label{posboun2}
\partial^+\tilde{B}_2(z):=\partial \tilde{B}_2(z)\cap\{y: (y-z)\cdot \tilde{R}(e_1)\geq M_0+M_0^\gamma\}.
\end{equation}
We then say that site $z\in M_0\mathbb Z^d$ is $M_0$-good, if
\begin{equation}\label{bgood}
\sup_{x\in \tilde{B}_1(z)}P_{x,\omega}\left[X_{T_{\tilde{B}_2(z)}}\in \partial^+\tilde{B}_2(z)\right]\geq\frac{1}{2},
\end{equation}
and $M_0$-bad otherwise. We have the following upper bound for
$M_0$-bad blocks:
\begin{lemma}
\label{controlbadbox}
Let $\gamma \in (5/9,1)$. Then, one has that
\begin{equation}
\label{estbb}
\limsup_{\substack{M_0\rightarrow\infty}}M_0^{5/4-(9/4\gamma)}\sup_{z\in M_0 \mathbb Z^d}\log \mathbb P[z\mbox{ is $M_0$-bad}]<0.
\end{equation}
\end{lemma}
\begin{proof}
For $z\in M_0\, \mathbb Z ^d$,

\begin{align}
  \nonumber
  \mathbb P [z\, \mbox{ is $M_0$-bad}]&=\mathbb P\left[\sup_{x\in \tilde{B}_1(z)}P_{x,\omega}\left[X_{T_{\tilde{B}_2(z)}}\notin \partial^+\tilde{B}_2(z)\right]>\frac{1}{2}\right] \\
  \label{eslecbad}
  &\leq 2 \left|\tilde{B}_1(z)\right|\sup_{x\in \tilde{B}_1(z)}P_x\left[X_{T_{\tilde{B}_2(z)}}\notin \partial^+\tilde{B}_2(z)\right].
\end{align}
Observe that for $x\in \tilde{B}_1(z)$, one has that
$\tilde{B}_2(z)$ is included in the closed Euclidean ball centered
at $x$ of radius $3\sqrt{d}M_0$. Therefore, recalling that
$\ell=l/|l|_2$ (cf. (\ref{rul})) one gets $P_x$-a.s.
\begin{equation*}
T_{\tilde{B_2}(z)}\leq T_{x\cdot l' +3\sqrt{d}M_0}^{l'}.
\end{equation*}
On the other hand, $P_x$-a.s. on the event
$\{X_{T_{\tilde{B}_2(z)}}\notin \partial^+\tilde{B}_2(z)\}$, one
has
\begin{equation*}
  \mbox{either  }(X_{T_{\tilde{B}_2(z)}}-x)\cdot \hat{v}\leq - \frac{M_0^\gamma}{2} \,\, \mbox{  or  } \left|\Pi_{\hat v}\left(X_{T_{\tilde{B}_2(z)}}-x\right)\right|_2\geq \frac{M_0^\gamma}{2},
\end{equation*}
where the notation is as in Proposition \ref{asymtrafluc}. As a
result, one gets
\begin{align}
\nonumber
\mathbb{P}\left[z \mbox{ is $M_0-$bad}\right]\leq c(d) M_0^d\Big(P_0\bigg[\sup_{0\leq n\leq T_{3\sqrt{d}M_0}^{\ell}}&|\Pi_{\hat{v}}(X_n)|_2\geq \frac{\hat{v}\cdot \ell}{4}M_0^\gamma\bigg]\\
\nonumber
+P_0\bigg[\sup_{0\leq n\leq T_{3\sqrt{d}M_0}^{\ell}}|\Pi_{\hat{v}}(X_n)|_2< \frac{\hat{v}\cdot \ell}{4}M_0^\gamma, \inf_{0\leq n \leq T_{3\sqrt{d}M_0}^{\ell}}&X_n\cdot \hat{v}\leq -\frac{M_0^\gamma}{2}\bigg]\Big)\\
\label{lemmabad}
\leq c(d)M_0^d\Big(P_0\bigg[\sup_{0\leq n\leq T_{3\sqrt{d}M_0}^{\ell}}|\Pi_{\hat{v}}(X_n)|_2\geq \frac{\hat{v}\cdot \ell}{4}M_0^\gamma\bigg]&+P_0\bigg[\widetilde{T}_{-\frac{M_0^\gamma \hat{v}\cdot \ell}{4}}^{\ell}<\infty\bigg]\Big),
\end{align}
where we used the inequality $X_n\cdot \ell \leq (X_n\cdot
\hat{v})\hat{v}\cdot \ell+\left|\Pi_{\hat{v}}(X_n)\right|_2$ to
obtain the rightmost term in the last line of (\ref{lemmabad}).
The claim follows now from Proposition \ref{asymtrafluc} and condition \textbf{(T)}$_\ell$.
\end{proof}
The general procedure is now to consider columns, constructed by
joining together boxes in direction $\hat{v}$. One then gathers
columns to form tubes. We next make precise the terms "column" and "tube"
by some further definitions. For $M>0$ and $M_0$ as above
(the relation between $M$ and $M_0$ will appear in
(\ref{choisedimen})), we attach to each $z\in M_0 \,\mathbb Z^d$,
the column
\begin{align}
\nonumber
Col(z)&=\left\{z'\in M_0 \mathbb Z^d:\, \exists j \in [0,J],
\,z'=z+jM_0 e_1\right\}, \,\,\mbox{ where}\\
\label{col}
&J \, \, \mbox{ is the smallest integer such that }\,JM_0
\hat{v}\cdot \frac{l}{|l|_2}\geq 3 M.
\end{align}
We choose $M_1>0$ an integer multiple of $M_0$ and define the tube
attached to $z\in M_0\mathbb Z^d$ by:
\begin{gather}
\nonumber
Tube(z)=\\
\label{tube}
\left\{z'\in M_0 \mathbb Z^d:\,\exists j_1,j_2\ldots,
j_d\in \left[0, \frac{M_1}{M_0}\right], \, z'=z+\sum_{i=2}^d j_iM_0
e_i \right\}.
\end{gather}
We stress that the key idea behind these definitions is the
following strategy: one way for the walk to escape from slab $U_M$
is to move to one of the \textit{bottom blocks} in $Tube(0)$ of an
appropriate column containing the greatest amount of \text{good
blocks} and then move along this column up to its top. Under the
choices that we will do later on, we will ensure that the walk
escapes from $U_M$ by the boundary side $\partial ^+ U_M$, see Figure \ref{fig2}.
\begin{figure}
  \centering
  \includegraphics[width=7cm]{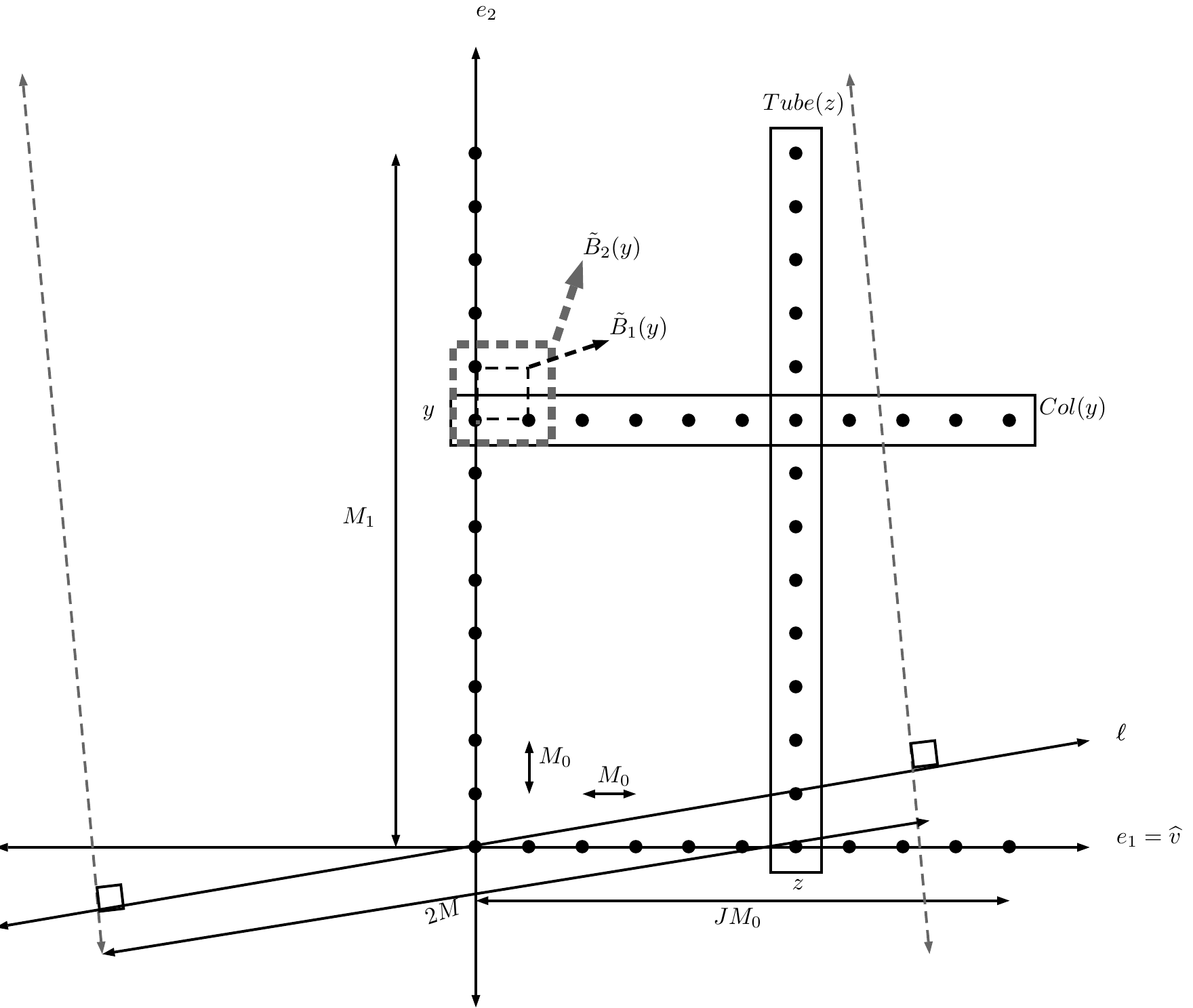}
  \caption{Schematic representation of definitions (\ref{col}), (\ref{tube}) and strategy to escape of $U_M$ by $\partial ^+ U_M$. Dashed lines delimitate set $U_M$ and if $y$ is a point in $Tube(0)$ with minimal amount of bad boxes on its column, the strategy is get to $y$ from $0$ and then use column on $y$ to escape by $\partial ^+ U_M$. The name columns comes from rotate counter clockwise this figure by a $\pi/2$ angle.}\label{fig2}
\end{figure}
\noindent
It will be convenient to introduce for $z\in M_0 \mathbb Z^d$, the
\textit{top of a tube }as:
\begin{equation}
\label{top}
Top(z)=\bigcup_{z'\in Tube(z)}\partial^+ \tilde{B}_2(z'+J M_0 e_1),
\end{equation}
along with the \textit{neighborhood of a tube} as:
\begin{gather}
\nonumber
V(z)=\\
\label{neitube}
\left\{x\in \mathbb Z^d:\, \exists y\in \bigcup_{\substack{z'
\in\, Tube(z),\\0\leq j\leq J }}\tilde{B}_1(z'+j M_0 e_1),\,
|x-y|_1\leq 3d M_1\right\}.
\end{gather}
We need a lower bound on the $P_{0,\omega}$-probability
for the event of reaching the top of a tube attached to the site
$0\in\mathbb Z^d$, before the walk exits from $V(0)$. To this end,
we introduce the random minimum number of $M_0-$bad boxes contained
in a column among columns in a tube as:
\begin{equation}\label{minbad}
n(z,\omega)=\min_{z'\in \, Tube(z)}\left\{\sum_{j=0}^J \,
\mathds{1}\textsubscript{$\{z'+M_0 j e_1$ is $M_0-$bad
$\}$}\right\}.
\end{equation}
Recalling our choice of $\kappa$ provided in (\ref{simplex}), the proof of Lemma 3.3 in \cite{Sz00} allows us to establish the next
\begin{lemma}
\label{quexplowbound}
There exists $c_{10}>0$ such that for any $z\in M_0\mathbb Z^d$ and
any
\begin{equation*}
  x\in \bigcup_{\substack{z' \in\, Tube(z),\\0\leq j\leq J }}\tilde{B}_1(z'+j M_0 e_1):=D(z),
\end{equation*}
one has
\begin{equation}\label{questimlow}
P_{x,\omega}\left[H_{Top(z)}<T_{V(z)}\right]\geq
(2\kappa)^{c_{14}\left(M_1+JM_0^\gamma+n(z,\omega)M_0\right)}\left(\frac{1}{2}\right)^{J+1}.
\end{equation}
\end{lemma}
In virtue of Lemma \ref{controlbadbox}, we now choose $\gamma\in
(5/9,1)$, such that:
\begin{equation}\label{chi}
\chi:=\frac{1-\beta}{1-\gamma}<\beta <1,
\end{equation}
and notice that such a choice is possible in view of assumption
(\ref{betacon}). We then choose
\begin{equation}\label{nu}
\nu>1-\gamma,
\end{equation}
and introduce for large $M$:
\begin{equation}\label{choisedimen}
M_0=\rho_1 M^\chi,\, M_1=\left[\rho_2M^{\beta-\chi}\right]M_0,\,
N_0=\left[\rho_3 M^{\beta-\chi}\right],
\end{equation}
where the constants $\rho_1,\,\rho_2,\, \rho_3 $, possibly depend
on $\kappa, |l|_2,\, |\hat{v}|_2, \,d,\, r, \, \delta$ and $c$
(cf.(\ref{aqe})). They are chosen so that for large $M$, the
following requirements:
\begin{align}\label{con1}
   &\min\left\{(2\kappa)^{c_{10}JM_0^\gamma}, \, (2\kappa)^{c_{10}M_1}, \, (2\kappa)^{c_{10}N_0M_0}, \,\left(\frac{1}{2}\right)^{J+1}\right\}\\
   \nonumber
   &>\exp\left(-\frac{c}{5}M^\beta\right),\\
  \label{con2}
&\frac{N_0}{3}>(J+1)\frac{(e^2-1)}{M_0^\nu}, \, \, \mbox{  and}\\
  \label{con3}
  &\mbox{any nearest neighbor path within $V(0)$, between $0$ and $Top(0)$,} \\
  \nonumber
  &\mbox{first exits $U_M$ through $\partial^+ U_M$.}
\end{align}
are satisfied.

\noindent
To see that such a choice is possible observe that it suffices to
take $\rho_1$ large enough and $\rho_2=\rho_3=c(10 \rho_1
c_{10}\log(1/(2\kappa)))^{-1}$, then (\ref{con1}) and (\ref{con3})
are satisfied for large $M$. As for (\ref{con2}), when $\beta < 1$,
it follows from the equality: $\beta -\chi=1-(1+\nu)\chi$.
\medskip

\noindent
Note that as a remark, the sites over which the environment events
$\{z:$ is $M_0$-good$\}$ depend, where $z$ runs over the
collection $(k_1M_0, \ldots, k_d M_0)$, with $k_i, \, i\in[1,d]$
non-negative integers, and $k_1+\ldots+k_d$ has a fixed parity;
they are at least a $|\,\cdot\,|_1$-distance of $M_0-2M_0^\gamma$
separated. Keeping this in mind, an application of
Bunyakovsky-Cauchy-Schwarz inequality gives
\begin{align}
\nonumber
  \mathbb P\left[n(0,\omega)>N_0\right]&=\mathbb P\left[\bigcap_{z'\in\, Tube(0)}\left\{\sum_{j=0}^{J}\mathds{1}\textsubscript{$\{z'+M_0 j e_1$    \,\,\mbox{ is  } \,\,$M_0-$bad $\}$}>N_0\right\}\right] \\
  \nonumber
  &\leq \mathbb P\left[\bigcap_{\substack{z'\in\, Tube(0)\\k_2+\ldots +k_d\,\, even}}\left\{\sum_{j=0}^{J}\mathds{1}\textsubscript{$\{z'+M_0 j e_1$    is $M_0-$bad $\}$}>N_0\right\}\right]^{\frac{1}{2}}\\
  \label{estnumbadbox}
  &\times\mathbb P\left[\bigcap_{\substack{z'\in\, Tube(0)\\k_2+\ldots +k_d\,\, odd}}\left\{\sum_{j=0}^{J}\mathds{1}\textsubscript{$\{z'+M_0 j e_1$    is $M_0-$bad $\}$}>N_0\right\}\right]^{\frac{1}{2}}.
\end{align}
Using the previous remark for the last two expressions on the rightmost
hand of (\ref{estnumbadbox}) along successive conditioning to
apply the mixing conditions (\ref{sma}) or (\ref{smg}), one gets:
%\begin{align}
%\label{fiestitenumbad}
%  \mathbb P&\left[n(0,\omega)>N_0\right]\leq \\
%\nonumber
 %  &\left(\exp(Ca)\sup_{z'\in \, Tube(0)}\mathbb P\left[\left\{\sum_{j=0}^{J}\mathds{1}\textsubscript{$\{z'+M_0 j e_1$    is $M_0-$bad $\}$}>N_0\right\}\right]\right)^{\left[\frac{M_1}{2M_0}\right]^{d-1}},
%\end{align}
\begin{align}
\label{fiestitenumbad}\mathbb P\left[n(0,\omega)>N_0\right] &\\
 \nonumber
  \leq \exp(Ca)\Big(\sup_{z'\in \, Tube(0)}\mathbb P\bigg[\sum_{j=0}^{J}\mathds{1}& \textsubscript{$\{z'+M_0 j e_1$    is $M_0-$bad $\}$}>N_0\bigg]\Big)^{\left[\frac{M_1}{2M_0}\right]^{d-1}}.
\end{align}
with the notation:
\begin{equation*}
  a=\sum_{x\in H , y \in T}e^{-g|x-y|_1},
\end{equation*}
where in turn for the mixing condition \textbf{(SM)}$_{C,g}$
(cf.(\ref{sma})), $H$, $T$ denote the sets:
\begin{align*}
H&=\partial^r\bigg\{y\in \mathbb Z^d:\, y\in \tilde{B}_2(z), \,\,
\mbox{ for some }\,\, z\in \tilde{B}_1\big(M_0 j
e_1+\Sigma_{\substack{2\leq i\leq d}}k_ie_i\big), \\
 &k_i\in\left[0,\frac{M_1}{M_0}\right],\,\,(k_2+\ldots+ k_d)-(d-1)\frac{M_1}{M_2}=1\,\, (mod\, 2),\,\,j\in [0,J]\bigg\}\\
T&=\partial^r\bigg\{x\in \mathbb Z^d:\,x\in \tilde{B}_2(z), \,\,
\mbox{ for some }\,\, z\in \tilde{B}_1\big(M_0 j
e_1+\Sigma_{\substack{2\leq i\leq d}}\frac{M_1}{M_0}e_i\big),
\\
&j\in [0,J] \bigg\},
\end{align*}
and for the mixing condition \textbf{(SMG)}$_{C,g}$ (cf.
\ref{smg}), the sets $H$ and $T$ will be switched to
\begin{align*}
H&=\bigg\{y\in \mathbb Z^d:\, y\in \tilde{B}_2(z), \,\, \mbox{ for
some }\,\, z\in \tilde{B}_1\big(M_0 j e_1+\Sigma_{\substack{2\leq
i\leq d}}k_ie_i\big), \\
 &k_i\in\big[0,\frac{M_1}{M_0}\big],\,\,(k_2+\ldots+ k_d)-(d-1)\frac{M_1}{M_2}=1\,\, (mod\, 2),\,\,j\in [0,J]\bigg\}\\
T&=\bigg\{x\in \mathbb Z^d:\,x\in \tilde{B}_2(z), \,\, \mbox{ for
some }\,\, z\in \tilde{B}_1\big(M_0 j e_1+\Sigma_{\substack{2\leq
i\leq d}}\frac{M_1}{M_0}e_i\big), \\
 &j\in [0,J] \bigg\}.
\end{align*}

\noindent
By means of a similar argument as the one of Lemma \ref{lemmaac},
one sees that for large $M_0$:
\begin{equation*}
  Ca\leq \exp\left(-\frac{gM_0}{2}\right),
\end{equation*}
and consequently
\begin{align}
\label{seestinumbad}\mathbb P\left[n(0,\omega)>N_0\right] &\\
 \nonumber
  \leq \exp(e^{-\frac{gM_0}{4}})\Big(\sup_{z'\in \, Tube(0)}\mathbb P\bigg[\sum_{j=0}^{J}\mathds{1}& \textsubscript{$\{z'+M_0 j e_1$    is $M_0-$bad $\}$}>N_0\bigg]\Big)^{\left[\frac{M_1}{2M_0}\right]^{d-1}}.
\end{align}
Let us now observe that arguing as in \cite{Sz00}, page 125, when
$Z$ is a Bernoulli random variable taking values onto $\{0,1\}$, with
success probability smaller than $M_0^{-\nu}$, then
$E[\exp(2Z)]\leq1+(e^2-1)/M_0^\nu$. As a result, restricting $j$ to
even or odd integers, we conclude from Chebyshev's inequality with
the help of: Lemma \ref{lemmabad}, successive conditioning, the
mixing conditions \textbf{(SM)}$_{C,g}$ or \textbf{(SMG)}$_{C,g}$ and the choice of
$\nu$ in (\ref{nu}), for large $M$
\begin{align}
\nonumber
&\sup_{z'\in \, Tube(0)}\mathbb P\left[\sum_{j=0}^{J}\mathds{1}\textsubscript{$\{z'+M_0 j e_1$    is $M_0-$bad $\}$}>N_0\right]\\
\nonumber
&\leq 2\exp\left(e^{-\frac{gM_0}{4}}\right)\exp\left(-N_0\right)\left(1+\frac{e^2-1}{M_0^\nu}\right)^{J+1}\\
\label{estatypqest}
&\leq 4\exp\left(-N_0+(J+1)\frac{e^2-1}{M_0^\nu}\right)\stackrel{(\ref{con2})}\leq \frac{1}{2}\,\exp\left(-\frac{N_0}{2}\right),
\end{align}
where we have assumed in turn that $M$ is large enough so that
\begin{equation*}
\exp\left(e^{-\frac{gM_0}{4}}\right)\leq 2.
\end{equation*}
Therefore, for large $M$
\begin{equation}\label{finalestatqe}
\mathbb P\left[n(0,\omega)>N_0\right]\leq \exp\left(-\frac{N_0}{2}\, \left[\frac{M_1}{2M_0}\right]^{d-1}\right).
\end{equation}
On the other hand, we have that on the event $\{n(0,\omega)\leq
N_0\}$:
\begin{equation*}
  P_{0,\omega}\left[X_{T_{U_M}}\cdot \frac{l}{|l|_2}\geq M\right]\stackrel{(\ref{con3})}\geq P_{0,\omega}\left[H_{Top(0)}<T_{V(0)}\right]\stackrel{(\ref{questimlow})-(\ref{con1})}>e^{-cM^\beta}.
\end{equation*}
Thus, one gets
\begin{equation*}
  \limsup_{M}\, M^{d(\beta-\chi)}\,\log \mathbb P\left[P_{0,\omega}\left[X_{T_{U_M}}\cdot \frac{l}{|l|_2}\geq M\right]\leq e^{-cM^\beta}\right]<0,
\end{equation*}
and the estimate (\ref{aqe}) follows by letting $\gamma$ vary
according to (\ref{chi}).
\end{proof}

\subsection{Proof of Theorem \ref{mth3}}
We now proceed with the proof of Theorem \ref{mth3}. The rough plan is to bound tails of the time $\tau_1$ and then we will apply Theorem 2 of \cite{CZ02}.
We begin with applying the previous atypical quenched estimate to obtain
controls on the tails of the approximate regeneration times
$\tau_1^{(L)}$. The precise statement will be the content of the
following:
\begin{proposition}
\label{tailestimate}
There exist constants $c_{11},\, c_{12}>0$ and $L_0\in |l|_1\,\mathbb N$,
so that for each $L\in |l|_1\, \mathbb N$ with $L\geq L_0$ and for
all $\alpha<1+\frac{4(d-1)}{13d+4}$:
\begin{equation}\label{taiest}
  P_0\left[\tau_1^{(L)}>u\right] \leq e^{-c_1\kappa^{L}(\log(u))^\alpha}+e^{-c_2(\log(u))^\alpha}.
\end{equation}
\end{proposition}
\begin{proof}
We pick an $\alpha \in \left(1,1+ \frac{4(d-1)}{13d+4}\right)$ and
consider for large $u$, the following choice of scales:
\begin{align}
\nonumber
  \Delta(u)&=\frac{1}{10\sqrt{d}}\, \frac{\log(u)}{\log\left(\frac{1}{\kappa}\right)}\,\, \mbox{ and }\,\,M(u)= N(u)\Delta(u), \\
  \label{scalestail}
  &\mbox{where } N(u)=\left[(\log(u))^{\alpha-1}\right].
\end{align}
To simplify notation we drop the dependence on $u$ for the remainder of
the proof.

\noindent
In virtue of Lemma \ref{basicexpest}, the claim will follow once we
can prove that:
\begin{equation}\label{tailrelexp}
\limsup_{u}\,\,\log(u)^{-\alpha}\,\log\left(P_0\left[T_{C_M}>u\right]\right)<0.
\end{equation}
Observe that for large $u$, one has (recall (\ref{cube}))
\begin{align}
\nonumber
  P_0\left[T_{C_M}>u\right]\leq \mathbb E\left[\forall x \in C_M, \, P_{x,\omega}\left[T_{C_M}\leq
  \frac{u}{\log(u)^\alpha}\right]\geq \frac{1}{2}, \, P_{0,\omega}\left[T_{C_M}>u\right]\right]\\
  \label{estitail1}
  +\mathbb P\left[\exists x_1 \in C_M, \, P_{x_{1},\omega}\left[T_{C_M}> \frac{u}{\log(u)^\alpha}\right]>\frac{1}{2}\right].
\end{align}
As a result of applying the Markov property, we see that
\begin{align}\nonumber
 &\mathbb E\bigg[\forall x \in C_M, \, P_{x,\omega}\big[T_{C_M}\leq
  \frac{u}{\log(u)^\alpha}\big]\geq \frac{1}{2},\,
  P_{0,\omega}\big[T_{C_M}>u\big]\bigg]\\
  \label{basicb}
  &\leq \left(\frac{1}{2}\right)^{[\log(u)^\alpha]}.
\end{align}
Therefore, in order to prove (\ref{taiest}), we need to obtain an
upper bound as above for the second term on the right-hand side of
(\ref{estitail1}). To this end, notice that when $x_1$ is such that
$P_{x_{1},\omega}[T_{C_M}>\frac{u}{\log(u)^\alpha}]>1/2$,
\begin{equation}
\label{entraandexit}
\frac{1}{2}\, \frac{u}{\log(u)^\alpha}\leq E_{x_{1}, \omega}\left[T_{C_M}\right]=\sum_{x\in C_M}\frac{P_{x_1,\omega}\left[H_x<T_{C_M}\right]}{P_{x,\omega}\left[\widetilde{H}_x>T_{C_M}\right]},
\end{equation}
where $\widetilde{H}_x:=\inf\{n\geq 1:\, X_n=x\}$, $H_x:=H_{\{x\}}$
(cf. (\ref{enandex})). To see how the last equality above is obtained, we calculate
\begin{align}
\nonumber
E_{x_{1}, \omega}\left[T_{C_M}\right]&=\sum_{n\geq
0}P_{x_1,\omega}\left[T_{C_M}> n\right]=\sum _{x\in
C_{M}}\left(\sum_{n\geq 0}P_{x_1,\omega}\left[T_{C_M}>
n,\,X_n=x\right]\right)\\
\label{basicm}
\sum_{x\in C_M}&E_{x_1,\omega}\left[\sum _{j=0}^{T_{C_M}}\mathds{1}_{\{X_j=x\}}\right]=\sum_{x\in C_M}E_{x_1,\omega}\left[\sum_{j\geq 1}\mathds{1}_{\{(H_x)_j<T_{C_M}\}}\right],
\end{align}
where we have defined $(H_x)_1=H_x$ for $x\in C_M$, and then by
recursion for $j>1$:
\begin{equation*}
 (H_x)_j=\widetilde{H}_x \circ \theta_{(H_x)_{j-1}}+(H_x)_{j-1}.
\end{equation*}
Applying the strong Markov property to the last term in
(\ref{basicm}), we get
\begin{align*}
E_{x_{1}, \omega}\left[T_{C_M}\right]=&\sum_{x\in C_M}P_{x_1, \omega}\left[H_x<T_{C_M}\right]\sum_{j\geq 1}P_{x,\omega}\left[\widetilde{H}_x<T_{C_M}\right]^{j-1} \\
=&\sum_{x\in C_M}\frac{P_{x_1, \omega}\left[H_x<T_{C_M}\right]}{P_{x,\omega}\left[\widetilde{H}_x>T_{C_M}\right]}. \\
  \end{align*}
Thus, coming back to (\ref{entraandexit}), one has that there
exists some $x_2\in C_M$ so that $\mathbb P$-a.s.
\begin{equation}
\label{entran}
  P_{x_2,\omega}\left[\widetilde{H}_{x_2}>T_{C_M}\right]\leq |C_M|\frac{2(\log(u))^\alpha}{u}
\end{equation}
holds, on the event in the second term on the rightmost side of (\ref{estitail1}).
Furthermore, notice that when $\omega\in \Omega$ is arbitrary, for $y:=x_2 \in C_M$ as in (\ref{entran}), and $x\in \mathbb Z^d$ a nearest neighbour lattice point to $y+\ell K$ with $0<K\leq\left[\frac{1}{3}\,
\frac{\log{u}}{\log\left(\frac{1}{2\kappa}\right)}\right]$ , by the elliptic assumption (\ref{simplex}) we have that
\begin{equation*}
P_{y, \omega}\left[\widetilde{H}_{y}>T_{C_M}\right]\geq
u^{-\frac{1}{3}}P_{x,\omega}\left[H_y>T_{C_M}\right].
\end{equation*}
Consequently for large $u$ we have $x\in C_M$ and
\begin{equation}
\label{uniell}
P_{x,\omega}\left[H_y>T_{C_M}\right]\leq \frac{1}{\sqrt{u}}.
\end{equation}
Thus, introducing the set
\begin{equation}\label{boundarydelta}
V_i= \partial\left\{y\in \mathbb Z^d:\,
y\cdot\frac{l}{|l|_2}<i\Delta\right\},\,\, \mbox{ for $i\in \mathbb
Z^d$,}
\end{equation}
we have that on the event
\begin{equation}\label{envevenaty}
\mathcal{E}=\bigcup_{x_1\in C_M}\left\{\omega\in \Omega:\,P_{x_1,\omega}\left[T_{C_M}>\frac{u}{(\log(u))^\alpha}\right]>\frac{1}{2} \right\}
\end{equation}
one can find $i_0\in [-N+1,N]$ and $x\in \,C_M \cap V_{i_0}$, such
that
\begin{equation}\label{implicationratyp}
  P_{x,\omega}\left[\widetilde{T}_{(i_0-1)\Delta}^l>T_{C_M}\right]\leq \frac{1}{\sqrt{u}}.
\end{equation}
Let us remark that in order to obtain (\ref{implicationratyp}), we have used (\ref{uniell}), the inequality: $2\Delta+d\leq \left[\frac{1}{3}\,
\frac{\log{u}}{\log\left(\frac{1}{2\kappa}\right)}\right]$ and the fact that for any $y\in \mathbb Z^d$ a closest point to $x-(2\Delta)\ell$, one has that
$$
P_{x,\omega}\left[\widetilde{T}_{(i_0-1)\Delta}^l>T_{C_M}\right]\leq P_{x,\omega}\left[H_{y}>T_{C_M}\right].
$$
It will be convenient to introduce for $i\in \mathbb Z^d$ the
random variables
\begin{equation*}
\mathfrak{X}_i:=\left\{
\begin{array}{ll}
-\log\left( \inf_{x\in C_M \cap V_i }
P_{x,\omega}\left[\widetilde{T}_{(i-1)\Delta}^l>T_{(i+1)\Delta}^l\right]\right)
&\quad{\rm if}\quad C_M\cap V_i\neq\emptyset,\\ 0&\quad{\rm
if}\quad C_M\cap V_i=\emptyset.
\end{array}
\right.
\end{equation*}
The next inequality is a consequence of induction along the strong
Markov property (cf. \cite{Sz00}, pp 128). For $i\in[-N+1, N]$ and
$x\in V_i$,
\begin{equation*}
P_{x,\omega}\left[\widetilde{T}_{(i-1)\Delta}^{l}>T_{C_M}\right]\geq
\exp\left(-\sum_{j=i} ^N\, \mathfrak{X}_i\right).
\end{equation*}
As a result from this last inequality and (\ref{implicationratyp}),
\begin{equation}\label{probofe}
  \mathbb P\left[\mathcal E\right]\leq 2N \sup_{i\in [-N+1, N]}\mathbb P \left[\mathfrak{X}_i\geq \frac{\log u}{2N}\right]
\end{equation}
Note that for $i\in \mathbb Z$, and $\nu>0$ one has
\begin{equation}\label{ineqxi}
\mathbb P\left[\mathfrak{X}_i>\nu\right]\leq |C_M|\mathbb P\left[P_{0,\omega}\left[X_{T_{U_{\Delta}}}\cdot \frac{l}{|l|_2}\geq \Delta\right]\leq e^{-\nu}\right].
\end{equation}
Therefore, using our version of an atypical quenched estimate given
in (\ref{atyquenchedestimate}), we get that whenever
\begin{equation*}
1> 2-\alpha\geq \frac{9d+4}{13d}\,\, \mbox{ (and thus $\alpha\leq
\frac{17d-4}{13d}$)}
\end{equation*}
one has
\begin{equation}
\label{probofe1}
\mathbb P\left[\mathcal E\right]\leq \exp\left(-\widetilde{c}(\log u)^{\chi}\right)
\end{equation}
for all $\chi<d(\frac{13}{4}(2-\alpha)-\frac{9}{4})$ and a suitable
constant $\widetilde{c}=\widetilde{c}(d,\kappa,l)$.

\noindent
In turn, the rightmost term in (\ref{probofe1}) is less than
\begin{equation*}
e^{-\widehat{c}(\log u)^\alpha}
\end{equation*}
for a positive constant $\widehat{c}$, whenever
$\alpha<\frac{17d}{13d+4}\leq \frac{17d-4}{13d}$. The proof is now
complete from this last argument as was mentioned after
(\ref{basicb}).
\end{proof}
We are ready to finish the proof of our main result.
\begin{proof}[Proof of Theorem \ref{mth3}]
We observe that Proposition \ref{tailestimate}, via \textit{layer cake decomposition} (cf. \cite{Ru87},
Chapter 8, Theorem 8.16) implies that there exists a
deterministic constant $M=M(L)$, such that
\begin{equation}
\label{Mltau1}
\mathbb P\left[ \frac{\bar E_0[(\kappa^L\,\tau_1)^3, D'=\infty\mid \mathfrak{F}_{0,L}]}{\bar P_0[D'=\infty \mid \mathfrak{F}_{0,L}]}>M\right]=0.
\end{equation}
The result of Theorem \ref{mth3} follows from the central limit theorem of \cite{CZ02}.
\end{proof}

\section{On Kalikow's Condition}
\label{seclln}
We will introduce in this section Kalikow's condition. We then prove that for a given $\ell\in \mathbb S^{d-1}$ the transient \textbf{(T)}$_\ell$ condition is satisfied whenever Kalikow's condition holds in the same direction. In the last part of this section, we will derive a ballistic strong law of large numbers, which is a slight extension of the main theorem in \cite{RA03}.
\subsection{$(T)$ is weaker than Kalikow's condition}
\begin{definition}
Kalikow's chain $(X_n)_{n\geq0}$ on a connected $V\subsetneq
\mathbb Z^d$ with $0\in V$ is the canonical Markov chain with state
space in $V \cup \partial V$, with transition probabilities given
by
\begin{equation*}
\widehat{P}_V(x,x+e):=\left\{
\begin{array}{ll}
\frac{E_0[\sum_{n=0}^{T_{V^c}}\mathds{1}_{\{X_n=x\}}\omega(x,e)]}
{E_0[\sum_{n=0}^{T_{V^c}}\mathds{1}_{\{X_n=x\}}]} &\quad{\rm
for}\quad x\in V\ {\rm and}\ |e|=1,\\ 1&\quad{\rm for}\quad x\in
\partial V\ {\rm and}\ e=0.
\end{array}
\right.
\end{equation*}
For $x\in V\cup \partial V$ we will denote by $\hat{P}_{x, V}$ and
$\hat{E}_{x,V }$ the law and expectation respectively of the
corresponding Kalikow's chain starting from $x$ with transition
probabilities as above. Setting the local drift
$\hat{d}_{V}(x)=\hat{E}_{x,V}[X_1-X_0]$ at site $x$ of this walk,
we say that Kalikow's condition is satisfied in direction $l\in
\mathbb R^d\backslash\{0\}$ and we denote this by \textbf{(K)}$_l$
if there exists a constant $\delta(l)>0$ such that
\begin{equation}
\label{Kalikow}
\inf_{x\in V, V} \hat{d}_{V}(x)\cdot l>\delta,
\end{equation}
where the infimum runs over all the connected strict subsets $V$ of
$\mathbb Z^d$, with $0\in V$.
\end{definition}

\noindent
We quote here the following result owed to S.
Kalikow \cite{Ka81}, which to some extend depicts the best known
property of Kalikow's chain.
\medskip
\begin{gather}
\nonumber
\mbox{Suppose that $\hat P_{0,V}-$a.s., $T_V$ is finite, then $P_0-$a.s. $T_V$ is also finite, }\\
\label{klemma}
\mbox{and $X_{T_V}$  has the same law under both $\hat P_{0,V}$ and $P_0$.}
\end{gather}
This property will be called as Kalikow's Proposition (see
\cite{Ka81}, Proposition 1 for a proof).

\medskip
Notice that when $|l|_2=1$, a straightforward application of
Cauchy-Schwarz inequality makes us see that the infimum in
(\ref{Kalikow}) is at most equal to $1$. In \cite{CZ01} was assumed
at the nestling example of Section 5 that this infimum is close to
$1$ for $l=e_1$, besides a conditional version of Kalikow's
condition. We will not need these further assumptions here.

\medskip
\noindent
Let us
note that for $n\geq0$,
\begin{equation}\label{martingale}
\mathfrak M_n ^V:= X_n-X_0-\sum_{0\leq j \leq n-1}\hat{d}_V(X_j)
\end{equation}
is a martingale for the canonical filtration of Kalikow's chain
$(X_n)_{n\geq 0}$ starting from $x\in V\cup \partial V$, with state
space in $V\cup \partial V$, where $V$ is a strict connected subset
of $\mathbb Z^d$ with $0\in V$. These martingales have increments
bounded in Euclidean norm by $2$, then Azuma-Hoeffding inequality
(see \cite{ASE92}, page 85) turns out that
\begin{equation}
\label{azumaine}
\hat P_{x,V}[\mathfrak M_n ^V \cdot w>A]\leq \exp\left(\frac{-A^2}{8n}\right) \mbox{ for } A>0, \, n\geq 0, \, |w|_2=1.
\end{equation}
We recall that under Kalikow's condition the process $(H_n)_{n\in \mathbb N}$, defined by (see \cite{Sz00}, pp 101-103 for a proof):
\begin{gather}
\label{Martingale1}
H_n:=\exp\left(-\eta X_n\cdot l\right)\\
\nonumber
\mbox{for all } \eta \in [0,\eta_0], \mbox{ where } \eta_0>0 \mbox{ depends on $\delta$},
\end{gather}
is a supermartingale under $\hat P_{x,V}$, for all strict connected subset $V$ of $\mathbb Z^d$ and $x\in V\cup\partial V$.

\vspace{0.2ex}
Letting $\ell\in \mathbb S^{d-1}$, the main result of this subsection comes in the next
\begin{proposition}
\label{kstrongt}
Assume \textbf{(K)}$_\ell$, then \textbf{(T)}$_\ell$ holds.
\end{proposition}
\begin{proof}
Assume condition \textbf{(K)}$_\ell$ and take $\delta>0$ as in the definition (\ref{Kalikow}).
In virtue of item $iii)$ of Lemma \ref{T}, we set $\mathrm r=2/\delta$ and for large $M$ we estimate (cf. (\ref{generalboxes}) for notation):
$$
P_0\left[X_{T_{B_{M,\mathrm{r}M,\ell}(0)}} \notin \partial^+ B_{M,\mathrm{r}M,\ell}(0) \right],
$$
where as usual the underlying rotation $R$ entering in the definition of the box $B_{M,\mathrm{r}M,\ell}(0)$ satisfies $R(e_1)=\ell$.

\vspace{0.2ex}
Notice that denoting $B_M$ the box $B_{M,\mathrm{r}M,\ell}(0)$, one has
\begin{align}
\nonumber
&P_0\left[X_{T_{B_{M,\mathrm{r}M,\ell}(0)}} \notin \partial^+ B_{M,\mathrm{r}M,\ell}(0) \right]\\
\nonumber
&\stackrel{(\ref{klemma})}= \hat P_{0, B_M}\left[X_{T_{B_M}}\cdot \ell < M\right]\leq \hat P_{0, B_M}\left[T_{B_M}>\mathrm{r}M\right]\\
\label{decomkt}
&+\hat P_{0, B_M}\left[T_{B_M}\leq \mathrm{r}M,\, X_{T_{B_M}}\cdot \ell < M \right].
\end{align}
We set $N=[\mathrm{r}M]$ and observe that $\hat P_{0,B_M}$-a.s. on $T_{B_M}>\mathrm{r}M$,
$$
\mathfrak M_N\cdot\ell<-M/2.
$$
Hence, using Azuma-Hoeffding inequality (\ref{azumaine}), we find that
\begin{equation}\label{esti1}
\hat P_{0, B_M}\left[T_{B_M}>\mathrm{r}M\right]\leq \exp\left(-\frac{M^2}{16N}\right).
\end{equation}
On the other hand, applying Chevyshev's inequality and the optional stopping theorem along the supermartingale in (\ref{Martingale1}), we get
\begin{align}
\nonumber
\hat P_{0, B_M}\left[T_{B_M}\leq \mathrm{r}M,\, X_{T_{B_M}}\cdot \ell < M\right]&=\hat P_{0, B_M}\left[ X_{T_{B_M}}\cdot \ell \leq -M\right]\\
\label{esti2}
&\leq \exp\left(-\eta  M\right).
\end{align}
Inserting (\ref{esti1}) and (\ref{esti2}) into (\ref{decomkt}) we complete the proof.
\end{proof}
The class of random environments studied in the present article extends the i.i.d. class. Alongside our ballisticity condition \textbf{(T)}$_{\ell}$ extends the previous i.i.d. condition $(T)|_\ell$ as well (cf. Theorem 1.1 of \cite{Sz02}). Recently in the framework of i.i.d. random environments, we have been able to prove the equivalence $(T)|_\ell \leftrightarrow (T')|_\ell$ (cf. Theorem 2.1 of \cite{GR18}). On the other hand, Sznitman in \cite{Sz03} has constructed ballistic walk examples satisfying $(T')|_\ell$ where Kalikow's condition breaks down, for all dimension $d\geq 3$. Thus, at least for dimensions $d\geq3$, condition \textbf{(T)}$_\ell$ is strictly weaker than \textbf{(K)}$_\ell$.

\subsection{Ballistic Regime under Kalikow's condition}
The next result can be thought as an alternative proof of the law of
large numbers in \cite{RA03} under Kalikow's condition, however a
\textit{slightly more general mixing condition} will be considered.
Precisely one has:
\begin{theorem}
\label{th1}
Let $C, g>0$. Assume that the RWRE fulfils conditions
\textbf{(K)}$_l$ and either \textbf{(SMG)}$_{C, g}$ or
\textbf{(SM)}$_{C,g}$ , then there exists a deterministic vector
$v\in \mathbb R^d\backslash\{0\}$, so that $P_0$-a.s.
\begin{equation*}
\lim_{n\rightarrow\infty}\frac{X_n}{n}=v,
\end{equation*}
with $v\cdot l>0$.
\end{theorem}
Let us begin by recalling the following (cf. \cite{Gu14}, Lemma 9 for a proof)
\begin{lemma}
\label{Guo1}
Let $a\in (0,1)$. Suppose that a sequence $(X_n)_{n\geq 1}$ of
nonnegative random variables satisfies
\begin{equation*}
a\leq\frac{dP[X_{n+1}\in \cdot|X_n, \ldots , X_1]}{d\mu}\leq a^{-1}
\end{equation*}
for all $n\geq 1$, where $P$ and $\mu$ are probability measures.
Setting $m_{\mu}=\int x d\mu(x)$, then $P$-a.s. one has that
\begin{equation*}
  a m_{\mu}\leq\liminf_{\substack{n\rightarrow\infty}} \sum_{k=1}^n X_k/n \leq \limsup_{\substack{n\rightarrow\infty}} \sum_{k=1}^n X_k/n\leq a^{-1} m_{\mu}.
\end{equation*}
\end{lemma}
The key result for our proof comes in the next proposition, where a
limiting but possibly vanishing velocity is proven. For $l\in \mathbb Z^d$ we will always assume \textbf{(K)}$_{l}$ (this is not a restriction, see Subsection \ref{Apre}) and either:
\textbf{(SM)}$_{C,g}$ or \textbf{(SMG)}$_{C,g}$. As a result of Proposition \ref{kstrongt}, for $L\in |l|_1\mathbb N$ we can construct the random variable $\tau_1^{(L)}$ along vector $l$.
\begin{proposition}
Assume \textbf{(K)}$_{l}$ an either: \textbf{(SM)}$_{C,g}$ or \textbf{(SMG)}$_{C,g}$. Then there exists $v\in \mathbb R^d$ deterministic, such that $P_0$-a.s.
\begin{equation}
\lim_{\substack{n\rightarrow\infty}} \frac{X_n}{n}\rightarrow v.
\end{equation}
\end{proposition}

\begin{proof}
We complete the unit vector $\frac{l}{|l|_2}$ to form an orthonormal base
of $\mathbb R^d$, which we will denote by
$\mathfrak{V}:=\{\frac{l}{|l|},w_2,\ldots, w_{d-1}\}$. We need the
following claim whose proof will be postponed:

\medskip
\noindent
\begin{center}
\textit{For all vector $w\in \mathfrak{V}$, there exist $\widehat{C}>0$ and $L_0\in |l|_1\mathbb N$ so that for all $L\geq L_0$  one has that}
\end{center}
\begin{equation}
\label{claim1}
\limsup_{\substack{n\rightarrow\infty}}\left|\kappa^L\frac{X_{\tau_n}\cdot w}{n}-\kappa^L \overline{E}_0[X_{\tau_1}\cdot w|D'=\infty]\right|\leq e^{-\widehat{C}L}.
\end{equation}
Assuming the previous claim we can now prove the proposition. Pick
a nondecreasing sequence $(k_n)_{n \geq 0}$, such that $$
\tau_{k_n}\leq n < \tau_{k_n+1}.
$$ By the very definition of the renewal structure we have
$\overline{P}_0$-a.s: $k_n$ goes to $\infty$ as
$n\rightarrow\infty$. Furthermore, with the help of Corollary
\ref{corren} we can use Lemma \ref{Guo1} to see that
\begin{equation*}
\limsup_{\substack{n\rightarrow\infty}}\left|\kappa^L\frac{\tau_n}{n}-\kappa^L \overline{E}_0[\tau_1|D'=\infty]\right|\leq e^{-\overline{C}L}
\end{equation*}
and by the claim
\begin{equation*}
\limsup_{\substack{n\rightarrow\infty}}\left|\kappa^L\frac{X_{\tau_n}}{n}-\kappa^L \overline{E}_0[X_{\tau_1}|D'=\infty]\right|_2\leq e^{-\overline{C}L},
\end{equation*}
for a suitable positive constant $\overline{C}$. Therefore, using
the decomposition
\begin{equation*}
\frac{X_n}{n}=\frac{X_{\tau_{k_n}}}{k_n}\frac{k_n}{n}+\frac{X_n-X_{\tau_{k_n}}}{n},
\end{equation*}
there exists a positive constant $C_6$, so that
\begin{equation*}
\limsup_{\substack{n\rightarrow\infty}}\left|\frac{X_n}{n}-\frac{\overline{E}_0[X_{\tau_1}|D'=]}{\overline{E}_0[\tau_1|D' =\infty]}\right|\leq e^{-C_6L}.
\end{equation*}
where we have used that
\begin{equation*}
\limsup_{\substack{n\rightarrow\infty}}\left|\frac{X_n-X_{\tau_{k_n}}}{n}\right|=0
\end{equation*}
which will be implied once we show that there exists $C_7>0$ such that
\begin{align}
\nonumber
\limsup_{\substack{n\rightarrow\infty}}\bigg|&\frac{\sum_{1\leq j\leq n}\sup_{0\leq i\leq \tau_1}|X_{i\circ \theta_{\tau_j}}-X_{\tau_j}|}{n}-\overline{E}_0[\sup_{0\leq i\leq \tau_1}|X_i| \,|D'=\infty]\bigg|\\
\label{totalosc}
&\leq e^{-C_7L}.
\end{align}
In order to prove (\ref{totalosc}), we apply Lemma \ref{Guo1} together with
Corollary \ref{corren} once again, to get
\begin{align*}
\limsup_{\substack{n\rightarrow\infty}}\bigg|&\frac{\sum_{1\leq j\leq n}\sup_{0\leq i\leq \tau_1}|X_{i\circ \theta_{\tau_j}}-X_{\tau_j}|}{n}-\overline{E}_0[\sup_{0\leq i\leq \tau_1}|X_i| \,|D'=\infty]\bigg|\\
&\leq 1-\exp\left(-2e^{(-gL)/4}\right),
\end{align*}
which implies the claim in (\ref{totalosc}). The proposition
follows now by letting (recall our notation $\tau_1=\tau_1^{(L)}$)
\begin{equation}
\label{vlimit}
v=\lim_{L\rightarrow\infty}\frac{\overline{E}_0[X_{\tau_1}|D'=\infty]}{\overline{E}_0[\tau_1|D'=\infty]},
\end{equation}
with the convention that $L$ in the limit runs over $\mathbb N |l|_1$. To see that such limit exists, notice
$$
\lim_{L\rightarrow\infty}\overline{E}_0[\tau_1|D'=\infty]=\sup_{L \in \mathbb N |l|_1}\overline{E}_0[\tau_1|D'=\infty]\in (0,\infty].
$$
Setting:
$$
\mathfrak T_1(L):=\kappa^L\overline{E}_0[\tau_1|D'=\infty],
$$
by virtue of Proposition \ref{remark}, we have two cases: $\lim_{L\rightarrow\infty}\mathfrak T_1(L)=\infty$ or there exists $k_3\in [k_2,\infty)$ such that $\lim_{L\rightarrow\infty}\mathfrak T_1(L)=k_3$. In the former case, using Corollary \ref{corexp} we have that $v=0$. In the second case we define for integer $n>1$
$$
v_n:=\frac{\overline{E}_0[X_{\tau_1^{(n|l|_1)}}|D'=\infty]}{\overline{E}_0[\tau_1^{(n|l|_1)}|D'=\infty]}.
$$
From the very definition of the renewal structure we have that for large integers $m>n$
\begin{align}
\nonumber
&\bigg|\overline{E}_0\big[X_{\tau_1^{(m|l|_1)}}-X_{\tau_1^{(n|l|_1)}}|D'=\infty\big]\bigg|_2\\
\label{estcauchy1}&\leq \overline{E}_0\bigg[\sup_{0\leq i\leq \tau_1^{((m-n)|l|_1)} } |X_n|_2|D'=\infty\bigg]
\end{align}
and,
\begin{equation}
\label{estcauchy2}\overline{E}_0\bigg[\tau_1^{(m|l|_1)}-\tau_1^{(n|l|_1)}|D'=\infty\bigg]
\leq\overline{E}_0\bigg[\tau_1^{((m-n)|l|_1)}|D'=\infty\bigg].
\end{equation}
Using both estimates (\ref{estcauchy1})-(\ref{estcauchy2}) and Proposition \ref{remark}, it is routine to prove that for large $m$ and $n$ with $m>n$,
$$
|v_m-v_n|\leq 2\kappa^{n|l|_1}.
$$
Therefore the limiting velocity in (\ref{vlimit}) exists.

\noindent
We now turn to prove claim (\ref{claim1}). Let $w\in \mathfrak{V}$
and set (with the notation $\tau_0=0$)
\begin{equation*}
Z_i=\kappa^L (X_{\tau_i}-X_{\tau_{i-1}})\cdot w
\end{equation*}
for integer $i\geq 1$. Using a coupling decomposition argument (cf.
\cite{CZ01}), we can enlarge the probability space where
the sequence $(Z_i)_{i\geq 1}$ is defined. We will still denote the new
probability measure by $\overline{P}_0$ in order to support the
following:
\begin{itemize}
  \item There exist two i.i.d. sequences $(\tilde{Z}_i)_{i\geq 1}$ and $(\Delta_i)_{i\geq1}$ such that $\tilde{Z_1}$ is distributed according to the distribution $\overline{P}_0[Z_1\in \cdot|D'=\infty]$, and $\Delta_1$ is Bernoulli distributed with values onto $\{0,1\}$ and success probability $\overline{P}_0[\Delta_1=1]=\exp(-\tilde{c}L)$, for some suitable and fixed constant $\tilde{c}>0$.
  \item There exists a third sequence $(W_i)_{i\geq1}$ so that for $i\geq 1$ one has that $\Delta_i$ is independent of $W_i$ and the $\sigma$-algebra $\mathcal G_i$ defined by
      \begin{equation*}
        \mathcal G_i=\sigma\left((Z_j)_{j\leq i-1}, (\Delta_j)_{j\leq i-1}\right) ,
      \end{equation*}
      with the convention that $\mathcal G_1$ is the trivial $\sigma$-algebra.
  \item In the new probability space, for integer $i\geq 1$ one has the decomposition:
  \begin{equation*}
    Z_i=\tilde{Z}_i(1-\Delta_i)+\Delta_i W_i.
  \end{equation*}
\end{itemize}
Therefore, one has on that large probability space
\begin{equation}\label{decom1}
  \frac{\kappa^LX_{\tau_{n}}\cdot w}{n}=\frac{\sum_{i=1}^n Z_i}{n}=\frac{\sum_{i=1}^n \tilde{Z}_i}{n}-\frac{\sum_{i=1}^n\tilde{Z}_i\Delta_i}{n}+\frac{\sum_{i=1}^n\Delta_i W_i}{n}.
\end{equation}
We are going now to estimate each one of the terms to the right of
(\ref{decom1}). The strong law of large numbers implies that
$\overline{P}_0$-a.s.
\begin{equation}\label{ineq1}
\frac{\sum_{i=1}^n \tilde{Z}_i}{n}\rightarrow \overline{E}_0[\tilde{Z}_1]=\overline{E}_0[\kappa^L X_{\tau_1}\cdot w|D'=\infty]
\end{equation}
and together with Corollary \ref{corexp}, $\overline{P}_0$-a.s. we
have
\begin{align}
\nonumber
 &\frac{\sum_{i=1}^n\tilde{Z}_i\Delta_i}{n}\rightarrow \overline{E}_0[\tilde{Z}_1\Delta_1]\leq \left(\overline{E}_0[(\kappa^L X_{\tau_1}\cdot w)^2|D'=\infty]\exp(-\tilde{c}L)\right)^{\frac{1}{2}}\\
 \label{ineq2}
 &\leq\exp(-\mathfrak{c}L),
\end{align}
for some positive constant $\mathfrak{c}$.

We next turn to bound from above the third expression on the right
most side of \ref{decom1}. This will be performed following a close
argument to the one of \cite{CZ01}, pp 894-895. Define
$\bar{W}_i:=\overline{E}_0[W_i|\mathcal G_i]$ and
$M_n=\Sigma_{i=1}^n(\Delta_i(W_i-\bar{W}_i))/i$, for integers $i$
and $n$ greater than $0$. Notice that $M_n$ is a $\mathcal
G_n$-martingale centered at $0$. We apply Burkholder-Gundy maximal
inequality (cf. \cite{Wi91}, Section 14.18) and Corollary
\ref{corexp} to get
\begin{equation*}
\overline{E}_0\left[\left|\sup_{n\geq 1}M_n\right|^{2}\right]\leq C_8 \overline{E}_0\left[\sum_{i\geq1}\frac{(\Delta_i(W_i-\bar{W}_i))^2}{i^2}\right]\leq \widetilde{C}_3
\end{equation*}
for some constants $C_8$ and $\widetilde{C}_3$. This implies that
$M_n$ almost surely converges to an integrable random variable.
Consequently, applying now Kronecker's lemma (cf. \cite{Wi91},
Section 12.7), one has that $\overline{P}_0$-a.s.
$H_n:=\Sigma_{i=1}^n (\Delta_i(W_i-\bar{W}_i))/n\rightarrow 0$.
Since $\Delta_i$ is independent of $\mathcal G_i$, using Corollary
\ref{corexp} and Jensen's inequality we get
\begin{align*}
  |\bar{W}_i|&\leq \left(\overline{E}_0[|W_i|^2|\mathcal G_i]\right)^{\frac{1}{2}}\\
  &\leq\left(\exp\left(e^{-(gL)/4}\right)\overline{E}_0[(\kappa^L X_{\tau_1}\cdot w)^2|D'=\infty]\right)^{\frac{1}{2}}\exp\left(\frac{\tilde{c}L}{2}\right)\\
  &\leq \overline{C}_4\exp\left(\frac{\tilde{c}L}{2}\right)
\end{align*}
where $\overline{C}_4>0$ is a constant. Hence
\begin{align}
\nonumber
  \sum_{i=1}^n \frac{\Delta_i \bar{W}_i}{n}&\leq \overline{C}_4\exp\left(\frac{\tilde{c}L}{2}\right)\sum_{i=1}^n\frac{\Delta_i}{n} \\
  \label{ineq3}
  &\stackrel{LLN}\rightarrow \overline{C}_4\, \exp\left(-\frac{\tilde{c}L}{2}\right).
\end{align}
Thus, combining (\ref{ineq1}), (\ref{ineq2}) and (\ref{ineq3}) we
have proven claim (\ref{claim1}).
\end{proof}
We need another auxiliary result in order to prove that the
limiting velocity $v$ is a non-vanishing one. Specifically,
Kalikow's condition admits a ballistic characterization (cf. \cite{SZ99}, pp 1861-1862 for a proof):

\begin{lemma}
For any finite connected set $U$ containing 0,
\begin{equation}\label{kball}
  E_0[T_U]\leq \frac{1}{\delta}\,E_0[X_{T_U}\cdot l]
\end{equation}
where $\delta$ is as in (\ref{Kalikow}) and $T_U$ is defined in
(\ref{enandex}).
\end{lemma}
We are now ready to prove Theorem \ref{th1}.
\begin{proof}[Proof of Theorem \ref{th1}]
Fixing $L\geq L_0$ with $L\in |l|_1 \mathbb N$, we consider for
$m\geq 0$, the nondecreasing sequence $k'_m$, $P_0$-almost surely
tending to $\infty$ as $m$ does (where as before, we use the
convention $\tau_0^{(L)}=0$), such that
\begin{equation*}
  \tau_{k'_m}^{(L)}\leq T_m^l<\tau_{k'_m+1}^{(L)}.
\end{equation*}
From the definitions of the sequence $(\tau_k^{(L)})_{k\geq1}$ (and
from now on, we drop the index $L$ for $\tau_k^{(L)}$ and
$X_{\tau_k}^{(L)}$), one has that $\overline{P}_0$-a.s.
\begin{equation*}
  l\cdot X_n< l\cdot X_{\tau_k}\leq l \cdot X_{n'},\,\, \mbox{       for      }\,\,0\leq n<\tau_k\leq n'.
\end{equation*}
Hence, for $m\geq 0$, $\overline{P}_0$-a.s.
\begin{equation}\label{einequ}
X_{\tau_{k'_m}}\cdot l\leq X_{T_m^l}\cdot l\leq X_{\tau_{k'_m+1}}
\cdot l
\end{equation}
and on the other hand, one has
\begin{equation}\label{mdeviation}
|X_{T_m}\cdot l-m|_2\leq \sup_{i\in[1,d]}|l_i|.
\end{equation}
Notice first that by Lemma \ref{Guo1} and Corollary \ref{corren}
one has that $\overline{P}_0$-a.s.
\begin{equation}
\label{guoimpl1}
\liminf_{\substack{m\rightarrow\infty}} \frac{k'_m}{X_{\tau_{k'_m}}\cdot l}\geq \exp\left(-e^{-(gL)/4}\right)\frac{1}{\overline{E}_0[X_{\tau_{1}}\cdot l|D'=\infty]},
\end{equation}
together with
\begin{equation}\label{guoimpl2}
\liminf_{\substack{m\rightarrow\infty}}\frac{\tau_{k'_m}}{k'_m}\geq \exp\left(-e^{-(gL)/4}\right)\overline{E}_0[\tau_1|D'=\infty].
\end{equation}
Moreover, a similar argument to the one in (\ref{totalosc}) gives the following upper
bound $\overline{P}_0$-a.s.
\begin{equation}\label{guoimpl3}
\limsup_{\substack{m\rightarrow\infty}}\frac{|(X_{\tau_{k'_m+1}}-X_{\tau_{k'_m}})\cdot l|}{m}=0.
\end{equation}
Hence, by the very definition of the sequence $k'_m$, estimates
(\ref{guoimpl1}), (\ref{guoimpl2}) and (\ref{guoimpl3}); we have
$\overline{P}_0$-a.s.
\begin{align}
\nonumber
  \liminf_{\substack{m\rightarrow\infty}} \frac{T_m^l}{m}&\geq\liminf_{\substack{m\rightarrow\infty}}\frac{\tau_{k'_m}}{m}=\frac{\tau_{k'_m}}{k'_m}\frac{k'_m}{X_{\tau_{k'_m}}\cdot l}\frac{X_{\tau_{k'_m}}\cdot l}{m} \\
  \nonumber
  &\geq \liminf_{\substack{m\rightarrow\infty}}\frac{\tau_{k'_m}}{k'_m}\liminf_{\substack{m\rightarrow\infty}} \frac{k'_m}{X_{\tau_{k'_m}}\cdot l}\liminf_{\substack{m\rightarrow\infty}} \frac{X_{\tau_{k'_m}}\cdot l}{m}\\
  \label{hh}
  &\geq\left(\exp\left(-e^{-(gL)/4}\right)\overline{E}_0[\tau_1|D'=\infty]\right)\left(\frac{\exp\left(-e^{-(gL)/4}\right)}{\overline{E}_0[X_{\tau_1}\cdot l|D'=\infty]}\right),
\end{align}
where to obtain the rightmost estimate in (\ref{hh}), we have used:
\begin{equation*}
  \liminf_{\substack{m\rightarrow\infty}}\frac{X_{\tau_{k'_m}}\cdot l}{m}=\liminf_{\substack{m\rightarrow\infty}} \left(\frac{X_{T_m^l}\cdot l}{m}-\frac{(X_{T_m^l} -X_{\tau_{k'_m}}) \cdot l}{m}\right)=1
\end{equation*}
which is satisfied, by virtue of (\ref{mdeviation}) and
(\ref{guoimpl3}). Furthermore, by an exhaustion of $\{y\in \mathbb
Z^d: y\cdot l<m\}$ by finite subsets of $\mathbb Z^d$, one sees
that applying Lemma \ref{kball} and Fatou's Lemma
\begin{equation*}
  \overline{E}_0\left[\liminf_{\substack{m\rightarrow\infty}}\frac{T_m^l}{m}\right]\leq \liminf_{\substack{m\rightarrow\infty}} \overline{E}_0\left[\frac{T_m^l}{m}\right]\leq \frac{1}{\delta}.
\end{equation*}
Therefore, Kalikow's condition implies that there exists a constant
$f=f(g,d,l,\delta)$ which does not depend on $L$ so that $$
\overline{E}_0[\kappa^L\tau_1|D'=\infty]\leq f.
$$ As a result $v:=\lim_{\substack{L\rightarrow\infty}}
\overline{E}_0[X_{\tau_1}|D'=\infty]/\overline{E}_0[\tau_1|D'=\infty]$
is a non-vanishing limiting velocity and furthermore, there exists
a constant $k_4>0$ such that $v\cdot l\geq k_4$ by Proposition
\ref{remark}.
\end{proof}

\section*{Acknowledgments}
I wish to thank Alejandro Ram\'{\i}rez my former PhD. advisor for
suggesting me this problem and some useful indications about the
paper writing. This work was almost completely done when I held a
postdoctoral position in the beautiful and amazing country of
Brazil at the Universidade Federal do Rio de Janeiro. Therefore I
thank Maria Eulalia Vares and Glauco Valle for the opportunity,
suggestions, corrections and hearing me in weekly held seminars
when I was developing this work. I am pretty sure that
\textit{voc\^{e}s estiverom tor\c{c}endo para que eu conseguisse.}
An anonymous referee has contributed with a detailed report and advised me about several
inaccuracies in a preliminary version, I am grateful for that.
Last but not least, my best thanks to Alain-Sol Sznitman, who
helped me to improve on the result of an earlier version in
Proposition \ref{tailestimate}.

\end{document}